\documentclass[12pt,reqno,twoside]{amsart}

\title
[Approximation for systems of linear forms via dynamics]
{Metric Diophantine approximation for systems of linear forms via dynamics}
\author{Dmitry Kleinbock}
\address{Brandeis University, Waltham MA
02454-9110 {\tt kleinboc@brandeis.edu}}
\author{Gregory Margulis}
\address{Yale University, 
   New Haven, CT 06520 {\tt margulis\@math.yale.edu}}
\author{Junbo Wang}
\address{Brandeis University, Waltham MA 02454-9110
{\tt junbo@brandeis.edu}}
\usepackage{amsmath}
\usepackage{amssymb}
\usepackage{epsfig}
\usepackage{color}

\font\sn = cmssi8 scaled \magstep0

\newif\ifdraft\drafttrue

\font\sn = cmssi8 scaled \magstep0

\draftfalse

\newcommand\name[1]{\label{#1}{\ifdraft{\sn [#1]}\else\ignorespaces\fi}}
\newcommand\bname[1]{{\ifdraft{\sn [#1]}\else\ignorespaces\fi}}

\newcommand\eq[2]{{\ifdraft{\ \tt [#1]}\else\ignorespaces\fi}\begin{equation}
\label{eq: #1}{#2}\end{equation}}

\newcommand{\under}[2]{\underset{\text{#1}}{#2}}

\numberwithin{equation}{section}

\newcommand\df{\stackrel{\mathrm{def}}{=}}

\newcommand{\R}{{\mathbb{R}}}
\newcommand{\Z}{{\mathbb{Z}}}
\newcommand{\br}{{\mathbb{R}}}
\newcommand{\bz}{{\mathbb{Z}}}

\newcommand{\N}{{\mathbb{N}}}

\newcommand{\vv}{{\bf{v}}}

\newcommand{\SL}{\operatorname{SL}}

\newcommand{\ggm}{G/\Gamma}

\newcommand{\Id}{\operatorname{Id}}
\newcommand{\diag}{\operatorname{diag}}

\newcommand{\x}{{\bf x}}
\newcommand{\vy}{{\bf y}}
\newcommand{\vx}{{\bf x}}
\newcommand{\vu}{{\bf u}}
\newcommand{\vd}{{\bf d}}

\newcommand{\ve}{{\bf e}}

\newcommand{\vf}{{\bf f}}
\newcommand{\vz}{{\bf z}}

\newcommand{\vp}{{\bf p}}
\newcommand{\vt}{{\bf t}}

\newcommand{\vq}{{\bf q}}
\newcommand{\til}{\widetilde}
\newcommand{\supp}{{\rm supp}}

\newcommand{\T}{{\mathbf{t}}}

\newcommand{\sm}{\smallsetminus}
\newcommand{\vre}{\varepsilon}
\newcommand\vwa{very well approximable}
\newcommand\vwma{very well multiplicatively  approximable}
\newcommand\da{Diophantine approximation}

\newcommand{\norm}[1] {\left\|{#1}\right\|}
      
\newcommand {\equ}[1]     {\eqref{eq: #1}}

\newcommand\ssm{\smallsetminus}
\newcommand\spr{Sprind\v zuk}
\newcommand\cag{$(C,\alpha)$-good}
\newcommand\nz{\smallsetminus \{0\}}
\newcommand\di{Diophantine}
\newcommand\de{Diophantine exponent}
\newcommand\hd{Hausdorff dimension}

\newcommand\dt{Dirichlet's Theorem}

\newcommand{\DI}{{\mathrm{DI}}}

\newcommand{\fa}{{\mathcal A}}

\newcommand{\ft}{{\mathcal T}}
\newcommand{\fs}{{\mathcal S}}
\newcommand{\fr}{{\mathcal R}}
\newcommand{\fw}{{\mathcal W}}
\newcommand{\const}{\operatorname{const}}

\newtheorem{thm}{Theorem}[section]
\newtheorem{lem}[thm]{Lemma}
\newtheorem{prop}[thm]{Proposition}
\newtheorem{cor}[thm]{Corollary}

\newcommand {\ignore}[1]  {}

\newcommand\mr{M_{m,n}}

\newcommand\amr{$Y\in M_{m,n}$}
\newcommand{\Span}{\operatorname{Span}}

\newcommand\vrn{\varnothing}

\newcommand{\vw}{{\bf{w}}}

\date{July 2009}

\begin{document}

     \begin{abstract}
The goal of this paper is to 
generalize the main results of \cite{KM} and subsequent
papers on metric \da\ with dependent quantities to the 
set-up of systems of linear forms.  In particular, we establish
`joint strong extremality' of arbitrary finite collection 
of smooth nondegenerate submanifolds of $\br^n$.
The  proofs are based on generalized
quantitative nondivergence estimates for translates of measures 
on the space
of lattices. 
\end{abstract}

\maketitle


\section{Introduction}
\label{intro}
The 
theory of simultaneous \da\ 
is concerned with the
following question: if $Y$ is an $m\times n$ real matrix (interpreted as a  system of
$m$
linear forms in $n$ variables),  how small, in terms
of the size of $\vq\in\bz^n$,  can be the distance from $Y\vq$ to  
$\bz^m$. This generalizes the classical theory of approximation of real numbers
by rationals, where $m = n = 1$. 

In the case of a single linear form ($m = 1$), or, dually, a single vector ($n = 1$),
significant progress has been made during recent years in showing that some important
approximation properties of vectors/forms happen to be generic with respect to certain
measures other than  Lebesgue measure. 
This circle of problems  dates back to the 1930s,
namely, to Mahler's work on 
transcendental 
numbers. 
In order to describe  more precisely Mahler's original problem, as well as 
subsequent results and conjectures,  let us introduce
some standard notions from the theory of \da.

Denote by $\mr$ the space of real matrices with  $m$ rows and $n$
columns. 
It follows from
\dt\ 
on simultaneous 
approximation 
\ignore{
states that for any \amr\ (viewed as a system
of $m$ linear forms in $n$ variables) 
and for any 
$t > 0$ there exist
$\vq = (q_1,\dots,q_n) \in \Z^n\nz$ and 
$\vp = (p_1,\dots,p_m)\in \Z^m$ satisfying the following system of inequalities:
\eq{dt}{
\|Y\vq - \vp\|
< e^{-t/m}
  \ \ \ \mathrm{and}  
\ \ \|\vq\|
\le e^{t/n}
\,.}
Consequently, 
}
that
for
any \amr\ there are infinitely many
  $\vq\in\bz^n$ such that
$\|Y\vq - \vp\| < \|\vq\|^{ - n /
m }$
 for some 
 $\vp\in\Z^m $
(here  $\|\cdot\|$ 
is given by 
$\|\x\| = \max_{i }|x_i|$.) 
On the other hand, if $\delta > 0$, the set of \amr\ such that 
there exist infinitely many
  $\vq\in\bz^n$ with
\eq{vwa}{
\|
Y\vq - \vp\| < \|\vq\|^{ - n /
m - \delta}\text{ for some }\vp\in\Z^m}
is null with respect to  Lebesgue measure $\lambda$. 
One 
 says that $Y$ is
{\sl \vwa\/} (abbreviated by VWA) if \equ{vwa} holds
for some positive $\delta$ and infinitely many   $\vq\in\bz^n$.
It follows  that 
the
set of VWA matrices has zero Lebesgue measure.
However its  \hd\ is equal to the dimension of $\mr$ \cite{Dodson}, so in this sense this set is rather big. 
Note also that by Khintchine's
Transference Principle, see e.g.~\cite[Chapter V]{Cassels}, 
$Y$ is VWA 
iff 
so is the transpose of $Y$. 

Let us now turn to a conjecture made by Mahler \cite{mahler}
 in 1932 and proved three decades
later  by \spr, see
 \cite{Sprindzuk-original, Sprindzuk-Mahler}. It states
 that 
 \ignore{given any $\delta > 0$, the set of
$x\in
\br$ such that
\eq{mahler}{
|P(x)| < h_P^{-(n+\delta)} \quad\text{for infinitely many
}P\in\bz[x]\quad\text{with }\deg P\le n}
has Lebesgue measure zero. Here
for a polynomial $$P(x) = a_0 + a_1x + \cdot + a_nx^n\in\bz[x]\,,$$
$h_P$ stands for the {\sl height\/} of $P$, that is, $h_P \df\max_{i =
0,\dots,n}|a_i|$.  Using the definition of the preceding
paragraph, it is not hard to see that Mahler's
Conjecture amounts to saying that}for $\lambda$-almost every
$x\in\br$,
the row vector $\vf(x) = (x,x^2,\dots,x^n)$ 
is not
VWA.
Sprind\v zuk's  proof of the above conjecture 
has led to the development of a new branch of number theory, 
the so-called `\da\ with dependent quantities'.  One of the goals of the theory has been showing that   certain smooth 
maps $\vf$ from open subsets of $\br^d$ to $\br^n$ are, in the terminology
introduced by Sprind\v zuk,
{\sl extremal\/}, that is, 
vectors\ \  $\vf(\vx)$ are not VWA
for $\lambda$-a.e.\ $\vx$ (the reader is referred to  \cite{BD}
for history and references).
Thus it seems  natural to
propose the following general 
problem: exhibit sufficient conditions
on a measure $\mu$ on $\mr$ (for example of the form $F_*\lambda$
where $F$ is a smooth map from an open subset of $\br^d$ to $\mr$)
guaranteeing that $\mu$ is {\sl extremal\/}, which by definition means that
$
\text{$\mu$-a.e.\ \amr\ is not VWA}
.
$
When $\mu = F_*\lambda$ for $F:\br^d\to \mr$, one can interpret this problem as studying $m$ maps $\br^d\to\br^n$ 
(rows of $F$)
simultaneously. Some special cases  were done by Kovalevskaya in the 1980s,
who used the terminology `jointly extremal' for the rows (or columns) 
of $F$ for which $F_*\lambda$ is extremal.

\medskip

The present paper, among other things, suggests possible solutions to this problem.  In fact 
this will be done in a stronger, multiplicative way. 
 For $\vx = (x_i)
 $ we let 
$$
\Pi(\vx) \df \prod_{i} |x_i|\quad \text{ and }\quad\Pi_{+}(\vx) \df
\prod_{i } \max(|x_i|, 1)\,.
$$
Then say that \amr\ is {\sl \vwma\/} (VWMA) if for some $\delta > 0$ there are infinitely many $\vq\in \Z^n$
such that \eq{hom}{ \Pi(Y\vq  -\vp) < \Pi_{+}(\vq)^{-(1+\delta)}
} for some $\vp\in\Z^m$.
Since $\Pi( Y\vq - \vp)$ is always not greater than $\|Y\vq - \vp\|^m$ and
$\Pi_{+}(\vq)
\le
\|\vq\|^n$ for
$\vq\in\bz^n\nz$, VWA implies VWMA. Still it can be easily shown that  Lebesgue-a.e.\ $Y$ is not VWMA\footnote{Also, generalizing Khintchine's Transference Principle one can show that $Y$ is  VWMA
iff so is the transpose of $Y$,
see a remark at the end of \S\ref{dioph}.}. 
Therefore one can ask for stronger sufficient conditions
on a measure $\mu$ on $\mr$ 
guaranteeing that it is {\sl strongly extremal\/}, that is,
$
\text{$\mu$-a.e.\ \amr\ is not VVWA}
.
$

An approach to this class of problems based on homogeneous dynamics
was developed in the paper \cite{KM},  which dealt with the case
$m = 1$. The problem of extending that approach to the matrix set-up
was  raised in \cite[\S 6.2]{KM} and then in \cite[\S9.1]{Gorodnik}. To state the main result of \cite{KM}, which  verified a conjecture
made by \spr\ in \cite{Sprindzuk-Uspekhi}, let us recall the following definitions. 
A smooth  map $\vf$ from  $U\subset\R^d$ to $\R^n$  is called 
{\sl $\ell$-nondegenerate at\/}  $\x\in U$ if   partial derivatives of
$\vf$ at $\x$
up to order $\ell$ 
span $\R^n$. We will say that  $\vf$  is {\sl nondegenerate at\/}  $\x$
if it is $\ell$-nondegenerate at $\x$ for some $\ell$, and that it is  {\sl nondegenerate\/}
if it is nondegenerate at $\lambda$-a.e.\ $\x\in U$. Here is 
the statement of  \cite[Theorem A]{KM}:

\begin{thm}\name{thm: strex} Let $\vf$ be a smooth nondegenerate
 map
from an
 open subset $U$ of $\br^{d}$ to $\br^n$. Then $\vf_*\lambda$  is 
 strongly extremal. \end{thm}

The goal of this paper is to  describe  a fairly large class of strongly extremal measures on $\mr$.
Here is an important special case of our general  results:

\begin{thm}\name{thm: strexanddi}
For every $i = 1,\dots,m$, let $\vf_i$ be a nondegenerate
 map
from an
 open subset $U_i$ of $\br^{d_i}$ to $\br^n$, and let
  \eq{f}{
F: U_1\times\dots\times U_m\to\mr,\quad(\vx_1,\dots,\vx_m) \mapsto \begin{pmatrix} \vf_1(\vx_1)\\
\vdots\\
\vf_m(\vx_m)\end{pmatrix}\,.
}
Then 
the pushforward of  Lebesgue measure on $U_1\times\dots\times U_m$ by $F$ is 
 strongly extremal.
\ignore{
 \item[(b)] there exists
$\vre_0$ such that for any $\vre < \vre_0$ one has
$${ F_*\lambda\big(\DI_\vre(\mathcal{T})\big) = 0\ 
\text{ for any unbounded } 
\mathcal{T}\subset\fa\,.
}$$
\end{itemize}
}
 \end{thm}

The case 
$d_1 = \dots = d_m = 1$,
i.e.\ that of
$n$ nondegenerate curves in $\R^m$, had been previously 
studied by Kovalevskaya \cite{Ko1, Ko2, Ko3}. A special case of the above theorem where  $U_1 = \dots = U_m$ and $\vf_1 = \dots = \vf_m$
is also of interest: it describes  approximation properties
of generic $m$-tuples of points (viewed as row vectors, or linear forms) on a given nondegenerate manifold. In this form 
the above statement had been conjectured earlier by Bernik (private communication).
We remark that recently V.\ Beresnevich informed us of an alternative approach allowing to prove
Theorem \ref{thm: strexanddi} when $\vf_1, \dots , \vf_m$ are real analytic. 
\medskip

The structure of the paper is as follows. 
In \S \ref{statements} we introduce the terminology needed to  state our
general result   (Theorem    \ref{thm: strexgeneral})
   of which  
Theorem \ref{thm: strexanddi}  
is a  special case. 
In \S\ref{dioph} we discuss a dynamical approach
to  \da\ problems and describe \di\ properties introduced above in the language of flows
on the space of lattices. Then in \S\ref{nondiv} and \S\ref{exp} we
present the main `quantitative nondivergence' measure  estimate  and  use it 
to state and prove a more precise
version of Theorem    \ref{thm: strexgeneral}.
\S\ref{indepvar} is devoted to proving Proposition \ref{prop: indepvar}, which explains why Theorem \ref{thm: strexanddi} follows from Theorem     \ref{thm: strexgeneral}.
Then the results of  \S\ref{nondiv}
 are used in \S\ref{moreexamples} for construction of  examples of extremal and strongly extremal measures not covered by Theorem~\ref{thm: strexgeneral}. Finally in the last section we mention several 
additional results and further open questions.  


\medskip

\noindent{\bf Acknowledgements.} Part of the
work was done during the  collaboration of D.K.\ and G.M.\ at  ETH-Zurich, the University
of Bielefeld and Yale University;  the
hospitality of these institutions is gratefully
acknowledged. 
This research was supported in part by NSF grants DMS-0239463,  DMS-0244406, DMS-0801064 and DMS-0801195.
Several results of this paper were  part 
of Ph.D.\ Thesis of the third named author \cite{junbo} defended at Brandeis in 2008. The authors are grateful
to Victor Beresnevich and the reviewer for useful remarks.

\section{
The main theorem}
\name{statements}

We now introduce  some   terminology needed to state
a more general version  of Theorem \ref{thm: strexanddi}.
Let $X$ be a 
metric space. 
If $x\in X$ and $r> 0$, we denote
by $B(x,r)$  the open ball of radius $r$ centered at $x$. If $B =
B(x,r)$  and $c > 0$, $cB$ will denote the ball $B(x,cr)$. For
$B \subset X$ and a real-valued function $f$ on $B$, let
$$\Vert f \Vert_{B} \stackrel{\mathrm{def}}{=}\sup_{x \in B} |f(x)|\,.
$$
%
If $\nu$  is  a measure on $X$
such that $\nu(B) > 0$,  define $\Vert f \Vert_{\nu,B} \df \Vert f
\Vert_{B\,\cap\,\supp\,\nu}$. All measures on metric spaces will be assumed to be Radon.

 If $D
> 0$ and $U\subset X$ is an open subset, let us say that a 
measure $\nu$ on $X$  is
{\sl $D$-Federer on $U$\/}  if one has $\nu(\frac13B) > \nu(B)/D$  for any ball $B\subset U$ centered
at $\supp\,\nu$.
This condition is often called `doubling'
in the literature; see \cite{KLW, MU} for examples and
references. A measure $\nu$ will be called {\sl Federer\/}
 if
for $\nu$-a.e.\ $ x\in X$  there exist a neighborhood $U$ of
$x$ and $D > 0$ such that $\nu$ is $D$-Federer on $U$.

\medskip
Given $C,\alpha > 0$ and open $U \subset X$, say that $f:U\to \R$ is
         {\sl $(C,\alpha)$-good on $U$ with respect to\/} a measure $\nu$
         if for any ball $B \subset U$ centered in $\supp\,\nu$
and any $\vre > 0$ one has 
$$\nu\big(\{ x\in B :
|f(x)| < \vre\}\big) \le C \left(\frac{\varepsilon}{\Vert
f\Vert_{\nu, B}}\right)^\alpha{\nu(B)}\,.
$$ This condition was
formally introduced in \cite{KM} for $\nu$ being Lebesgue
measure on $\R^d$, and in \cite{KLW} for  arbitrary $\nu$.
%
%
If $\vf = (f_1,\dots,f_N)$ is a map from $U$ to $R^N$, following \cite{dima tams}, we will  say that
a pair $(\vf,\nu)$ is  {\sl good\/}
 if for $\nu$-a.e.\ $x$
   there exists  a
neighborhood $V$ of $x$  
such that
any linear 
combination of $1,f_1,\dots,f_N$  is $(C,\alpha)$-good on $V$ with
respect to $ \nu$. 
\ignore{
For fixed $C,\alpha$, we will say that  $(\vf,\nu)$ is  {\sl
\cag\/}   if for $\nu$-a.e.\ $x$
   there exists  a
neighborhood $U$ of $x$  
such that $(\vf,\nu)$ is 
 $(C,\alpha)$-good  on $U$, and that $(\vf,\nu)$ is  {\sl good\/} if  for $\nu$-a.e.\ $x$
   there exists  a
neighborhood $U$ of $x$  and $C,\alpha > 0$ 
such that $(\vf,\nu)$ is 
 $(C,\alpha)$-good  on $U$.}

Here is another useful definition: 
$(\vf,\nu)$ is said to be {\sl nonplanar\/}
if for any ball $B$
with $\nu(B) > 0$, the restrictions of $ 1,f_1,\dots,f_{N}$
to $B\,\cap \,\supp\,\nu$ are linearly independent over $\R$; in
other words,  $\vf(B \,\cap\, \supp\,\nu)$  is not contained in
any proper affine subspace  of $\R^N$. 

Important examples of good and nonplanar pairs $(\vf,\nu)$ 
are $\nu = \lambda$ (Lebesgue
measure on $\R^d$) and $\vf$ smooth and nondegenerate. In this
case the fact that $(\vf,\lambda)$ is  
good follows from  \cite[Proposition 3.4]{KM}, and
nonplanarity is immediate. 
In \cite{KLW} a class of {\sl friendly\/} measures was introduced:
  a measure $\nu$ on $\R^n$ is  
friendly if and only if it is Federer and the pair $(\Id,\nu)$
is good and nonplanar; many examples of those can be found in \cite{KLW, Urbanski, Urbanski-Str}. 
In the paper \cite{KLW} the approach to metric \da\ developed in \cite{KM}
has been extended to maps and measures satisfying the conditions described above.
One of its main results is the following theorem \cite[Theorem 4.2]{dima pamq}, implicitly contained in  \cite{KLW}: let $\nu$ be a Federer measure on
$\br^d$, $U\subset\br^d$ open, and $\vf:U\to\br^n$ a continuous map such that $(\vf,\nu)$ is 
good and
nonplanar; then $\vf_*\nu$
is strongly extremal. 
\ignore{
Another relevant  result is  \cite[Theorem 1.5]{KW} which says the following:
let 
$\nu$ be a  $D$-Federer 
measure  on
$\R^d$, $U\subset\br^d$ open, and   $\vf:U\to\br^n$  continuous such that 
the pair $(\vf,\nu)$ is 
$(C,\alpha)$-good and
nonplanar; then  $ \vf_*\nu\big(\DI_\vre(\mathcal{T})\big) = 0$
for any unbounded $ 
\mathcal{T}\subset\fa$ and any $\vre < \vre_0$, where
$\vre_0$ depends only on $d,n,C,\alpha,D$.
}

Our goal in this paper is to  replace $\R^n$ with $\mr$ in the above statements. 
For this, given 
$Y = (y_{i,j})\in\mr$ and subsets 
$I =
\{i_1,\dots,i_{r}\}\subset \{1,\dots,m\}$ and $J =
\{j_1,\dots,j_{r}\}\subset \{1,\dots,n\}$ of equal cardinality and with $i_1 <\dots <i_{r}$ and $j_1<\dots<j_{s}$,
we define \eq{def yij}{
y_{I,J} \df\left|\begin{matrix}
                              y_{i_{1},j_{1}} & \cdots & y_{i_{1},j_{r}} \\
                               \cdots & \cdots & \cdots \\
                               y_{i_{r},j_{1}} & \cdots & y_{i_{r},j_{r}}
                              \end{matrix}\right|\,,\text{ with the convention }y_{\varnothing, \varnothing} = 1
                     \,.
}
Denote by
\eq{def n}{N \df {{m+n}\choose{ m}}  - 1
} 
the number of different square submatrices of an $m\times n$ matrix, and
 consider  the map $\vd:\mr\to\br^N$ given by
$$
\vd(Y) \df  \big(y_{I,J}\big)_{I\subset  \{1,\dots,m\},\  J\subset  \{1,\dots,n\},\  0 <|I| = |J| \le \min(m,n)}\,.
$$
In other words, $\vd(Y)$ is a vector whose coordinates are determinants of all possible square submatrices of $Y$ (the order in which they appear does not matter).
\medskip

At this point we can state  the
 main result of the paper:

\begin{thm}\name{thm: strexgeneral} Let $\nu$ be a Federer measure on
$\br^d$, $U\subset\br^d$ open,  and $F:U\to\mr$ a continuous map such that $(\vd \circ F,\nu)$ is 
{\rm (i)} good and
{\rm (ii)} nonplanar.  Then $F_*\nu$
is strongly extremal.
\end{thm}

\ignore{
\begin{thm}\name{thm: digeneral} For any\,
$d,m,n\in\N$ and   
$\,C,\alpha,D > 0$ there exists
$\vre_0 =  \vre_0(d,m,n,C,\alpha,D)$ with the following property.
 Let 
 $\nu$ be a  
measure  on
$\R^d$, $U\subset\br^d$ open, and     $F:U\to\mr$ continuous. 
Assume that $\nu$ is  $D$-Federer,  and $(\vd \circ F,\nu)$ is 
{\rm (i)}  $(C,\alpha)$-good and
{\rm (ii)}  nonplanar.
Then for any $\vre < \vre_0$
\eq{result general}{ F_*\nu\big(\DI_\vre(\mathcal{T})\big) = 0\ 
\text{ for any unbounded } 
\mathcal{T}\subset\fa\,.
}
\end{thm}
}

Note that  if  $\min(m,n) = 1$, $
 \vd \circ F$ coincides with $F$, and $N$ is equal to $\max(m,n)$; thus  \cite[Theorem 4.2]{dima pamq}
cited above  is a special case  of Theorem  \ref{thm: strexgeneral}.
If $\min(m,n) > 1$, 
the assumptions (i) and (ii) above can be verified for a wide variety of examples.
For instance, when a map $F:U\to\mr$ is real analytic and $\nu$ is Lebesgue measure, 
assumption (i) of both theorems is satisfied (this follows from the results of \cite{KM}
and \cite{gafa}). And  if $F$ is differentiable and $\nu = \lambda$, both (i) and (ii) would  follow from
an assumption that the map $\vd\circ F:\br^d\to\br^N$ is nondegenerate. We explain this in more detail in 
\S\ref{indepvar}, where we also    prove

\begin{prop}\name{prop: indepvar} 
Let $F: U_1\times\dots\times U_m\to \br^n$
be as in   Theorem \ref{thm: strexanddi}.
Then the pair $(\vd\circ F,\lambda)$ is good and nonplanar.
\end{prop}

In view of the above proposition and since Lebesgue measure is  Federer, Theorem  \ref{thm: strexanddi} follows from 
Theorem  \ref{thm:  strexgeneral}.

\medskip

We remark that the assumptions (i) and (ii) of Theorem
\ref{thm: strexgeneral} 
are not the most general possible; in particular, assuming (i) one can
establish 
necessary and sufficient conditions for the extremality and strong extremality of $F_*\nu$, see Theorem   \ref{thm: criterion}. 
Some examples of extremal and strongly extremal measures not covered by
Theorem
\ref{thm: strexgeneral}  are discussed in \S\ref{moreexamples}.

\section{Diophantine approximation and flows on homogeneous spaces}
\label{dioph}



From now on we will let $k = m+n$ and put $G=\SL_{k}( \R), \ \Gamma=\SL_{k}( \Z)$ and $\Omega = \ggm$.
Note that $\Omega$
is naturally identified with the space of unimodular lattices in
$\R^{k}$ via the correspondence $g\Gamma \mapsto g\Z^{k}$.
Define
$$
u_Y \stackrel{\mathrm{def}}{=} \left(
\begin{array}{ccccc}
I_{m} & Y \\ 0 & I_{n} 
\end{array}
\right), \ \ \ \Lambda_Y \stackrel{\mathrm{def}}{=} 
u_Y \Z^{k}\,,
$$
where $I_{\ell}$ stands for the $\ell\times \ell$ identity matrix.
To highlight the relevance of the objects defined above to the \di\ problems considered in the introduction,
note that 
$$
 \Lambda_Y = \left\{\begin{pmatrix} Y\vq - \vp\\\vq\end{pmatrix} : \vp\in\Z^m,\  \vq\in\Z^n\right\}\,.
$$
The main theme of this section is a well known restatement of  \di\ properties of $Y$ in terms of
behavior of certain orbits of $\Lambda_Y$ on $\Omega$. 
Let us denote by $\fa
$
the set of $k$-tuples $\vt = (t_1,\dots,t_{k})\in \R^{k}$
such that
\eq{sumequal}{
t_1,\dots,t_{k} > 0
\quad \mathrm{and}\quad 
\sum_{i = 1}^m t_i =\sum_{j = 1}^{n} t_{m+j} \,.
} 
To any $\vt\in\fa
$ 
let us associate the diagonal
matrix $$g_\vt \df \diag(e^{t_1}, \ldots,
e^{ t_m}, e^{-t_{m+1}}, \ldots,
e^{-t_{k}})\in G\,.$$
If $\ft$ is a subset of $\fa$, we let $g_\ft \df \{g_\vt : \vt\in\ft\}$.
We are going to consider $g_\ft$-orbits of lattices $\Lambda_Y$.
The two most important special cases will be $\ft =
\fa$ and $\ft = \fr$, where
\eq{def r}{\fr
 \df \left\{\left(\tfrac t m,\dots,\tfrac t m,\tfrac t n,\dots,
\tfrac t n\right)
: t > 0\right\}} is 
the `central ray' in $\fa
$. 
Also it will be convenient to use the following notation: for $ \vt\in\mathcal{A}$, we will denote
\eq{def t}{t = 
\sum_{i = 1}^m t_i =\sum_{j = 1}^{n} t_{m+j}\,, }
so that whenever $\vt$ and $t$ appear in the same formula, \equ{def t} will be assumed.
Clearly one has  $t \ge \|\vt\| \ge t/\min(m,n)$. Note also that this agrees with
the notation of \equ{def r}.

 Given $\varepsilon>0$, consider
\begin{equation*}
\begin{split}
K_{{\varepsilon}} &\stackrel{\mathrm{def}}{=} \big\{\Lambda\in\Omega \bigm| \Vert 
\vv \Vert \geq {\varepsilon} \quad \forall\, \vv \in
\Lambda \sm \{ 0 \}\big\},
\end{split}\end{equation*}
i.e.\ the collection of all unimodular
lattices in $\R^{k}$ which contain no nonzero vector 
of norm smaller than
$\varepsilon$). 
By Mahler's compactness criterion (see e.g.\  \cite[Chapter
10]{Rag}), each $K_{{\varepsilon}}$ is compact.
It has been observed in the past\footnote{See also \cite{Dani} where it is proved that $Y$ is badly approximable iff $g_\fr\Lambda_Y$ is bounded.}
that the existence of infinitely many solutions of inequalities \equ{vwa} and \equ{hom} corresponds to an unbounded sequence of excursions of certain trajectories outside of
the increasing family of compact subsets described above
 --  roughly speaking, to the trajectories growing with certain rate.
To make this specific, given 
$\mathcal{T}\subset \fa
$ and a lattice $\Lambda\in\Omega$, say that the trajectory $g_{\mathcal{T}}  \Lambda$ 
{\sl has linear growth\/}
if there exists $\gamma > 0$ such that
$$g_{\T} \Lambda \notin K_{e^{-\gamma t}}\text{ for an unbounded set of }\vt\in\mathcal{T}\,.
$$
(The terminology is justified by the fact that for small $\vre$, the diameter of $K_{{\varepsilon}}$
is bounded from both sides by $\const\cdot \log(1/\vre)$.)

\medskip

The next proposition gives the desired correspondence between approximation and dynamics:

\begin{prop}\name{KM - rate of escape condition} Let \amr.

\begin{itemize}
\item[(a)]
$Y$ is VWA $\Leftrightarrow$ $g_{\mathcal{R}}  \Lambda_Y$  has linear growth; 
 \item[(b)]
$Y$ is VWMA $\Leftrightarrow$  $g_{\mathcal{A}}  \Lambda_Y$  has linear growth.
\ignore{
\item[(c)] for an unbounded $\ft\subset \fa$, $Y\in\DI_\vre(\mathcal{T})$ $\Leftrightarrow$ $ \exists\,T > 0$ such that $g_\vt \Lambda_Y   \in K_\vre$
  for all $\vt\in\ft$ with $t > T$; in other words,
  \eq{liminf}
{\DI_\vre(\mathcal{T}) = \bigcup_{T> 0}\quad\bigcap_{\vt\in\mathcal{T} ,\,t > T}
\{Y : 
g_\vt \Lambda_Y   \notin K_{\vre}\}\,.}
}
\end{itemize}
\end{prop}

%

Part (a) is a special case of \cite[Theorem 8.5]{loglaws}. 
Part (b), more precisely, its `$\Rightarrow$' direction, has been worked out
in \cite{KM} and \cite{KLW} in the cases $m = 1$ and $n = 1$ respectively (converse direction is easier and was not required for applications). See also \cite[Theorem 9.2]{loglaws} for a related statement.
The proof of the general case of (b) combines the argument of the aforementioned papers; 
to make this paper self-contained we include the proof of both directions.

\begin{proof}[Proof of Proposition \ref{KM - rate of escape condition}(b)] Start with the `if' part.
Suppose  there exists $\gamma > 0$ and an unbounded subset $\mathcal{T}$ of $\mathcal{A}$ such that whenever $ \vt\in\mathcal{T}$, 
for some $(\vp,\vq)\ne 0$ one has \eq{part 1}{e^{t_{i}}|Y_i\vq - p_i| <  e^{-\gamma
t },\quad i = 1,\dots,m\,,} and 
\eq{part 2}{e^{ -t_{m+j} }|q_{j}| < e^{-\gamma t},\quad j = 1,\dots,n\,.} 
We need to prove that $Y$ is VWMA. 
Let $\ell$ be the number of nonzero components of $\vq$. 
(Note that $\vq\ne 0$,  otherwise from \equ{part 1} it would follow that $\vp = 0$,
hence $(\vp,\vq) =0$.)
Multiplying the inequalities in \equ{part 2} corresponding to $q_i \ne 0$
one gets  $e^{-t}\Pi_{+}(\vq) < e^{-\ell\gamma t}$, or
$\Pi_{+}(\vq) < e^{(1-\ell\gamma) t}$. On the other hand, after multiplying inequalities from \equ{part 1} 
one has
$e^{t} \Pi( Y\vq - \vp ) < e^{-n\gamma t}$,
or 
\eq{concl}{
\Pi( Y\vq - \vp )\leq e^{-(1+n\gamma) t} = (e^{(1-\ell\gamma) t})^{-\frac{1+n\gamma}{1-\ell\gamma}}
< \Pi_{+}(\vq)^{-\frac{1+n\gamma}{1-\ell\gamma}}\,.}
Therefore, \equ{hom} is satisfied with some positive $\delta = \delta(\gamma)$. Finally observe that
$Y$ is obviously VWMA if $Y_i\vq \in\bz$ for some $i$ and $\vq\in\bz^n\nz$: indeed, it suffices to take
integer multiples of $\vq$ to satisfy \equ{hom}. Otherwise, taking $\vt\to\infty$ in $\mathcal{T}$ 
we get infinitely many  $\vq$ for which \equ{concl}, and hence \equ{hom}, holds. 

\medskip

For the other direction, let us
prove two auxiliary lemmas.

\begin{lem}\label{lem: aux1}
Let \amr\ be VWMA.
Then there exists $\,\delta > 0$ for which there are infinitely
many solutions $\vp \in \Z^m$, $\vq \in \Z^n\nz$ to
  \equ{hom} in addition satisfying
\begin{equation}
\label{eq:equation with new condition} \norm {  Y \mathbf q
- \mathbf p } < \Pi_{+}(\mathbf q)^{ - \delta/m  }\,.
\end{equation}
\end{lem}


\begin{proof} We follow the argument of \cite{KLW}. Choose $\delta_0>0$
so that we have
\begin{equation}
\label{eq: nvwma} \Pi( Y\vq-\vp) < \Pi_{+}(\vq)^{ -(
1 + \delta_0) }\,,
\end{equation}
for infinitely many $\vp \in \Z^m$, $\vq \in \Z^n$.
Let $\vp, \, \vq$ be a solution to (\ref{eq: nvwma}), and let $$q
\df [ \Pi_{+}(\vq)^{\frac{\delta_0}{m+n+1 } } ]\,.$$ 
We can assume that $\Pi_{+}(\vq)$ is large enough so that
$q
\ge   \frac12\Pi_{+}(\vq)^{\frac{\delta_0}{m+n+1 } }$.
For every $\ell \in
\left\{1 , \dots , q+1 \right\}$ set
\begin{equation*}
\mathbf v_{\ell } \df  \ell Y \mathbf q\, \bmod 1
\end{equation*}
(here the fractional part is taken in each coordinate). Since
$\left\{ \mathbf v_1 , \dots , \mathbf v_{ q+1 } \right\}$ are $q
+ 1$ points in the unit cube $[ 0 , 1 )^m$, there must be two
points, say $\mathbf {v}_ i, \, \mathbf{v}_{j}$, with $1 \leq i
< j \leq q+1$, such that
\begin{equation}
\label{eq:we will use this soon} \norm { \mathbf v_{ i } - \mathbf
v_{ j } } \le q ^{-\frac{1}{m}} \le
(\tfrac12\Pi_{+}(\vq)^{\frac{ \delta_0}{m+n+1} })^{-1/m} = 2^{1/m}\Pi_{+}(\vq)^{-\frac{ \delta_0}{m(m+n+1)} }.
\end{equation}

We set $\bar{\vq} \df (j-i) \vq$ and choose $\bar{\mathbf{p}} \in
\Z^m$ to be an integer vector closest to $ Y\bar{\vq}$. Note
that 
\eq{notethat}{\Pi_{+}(\bar{\vq}) \leq (j-i)^n\Pi_{+}(\vq) \le \Pi_{+}(\vq)^{\frac{n\delta_0}{m+n+1 } + 1}
\,.}
Then by inequality
(\ref{eq:we will use this soon}),
\begin{equation*}
\begin{aligned}
\norm {  Y\bar{\vq}  - \bar{\mathbf{p}} } &\le
2^{1/m}\Pi_{+}(\vq)^{\frac{ \delta_0}{m(m+n+1)} } \under{\equ{notethat}}\le 2^{1/m}
\Pi_{+}(\bar{\vq})^{-\frac{m+n+1 }{m+n+1 + n\delta_0}\frac{ \delta_0}{m(m+n+1)}}\\ &=
2^{1/m}\Pi_{+}(\bar{\vq})^{-\frac{\delta_0}{m(m+n+1 + n\delta_0)}}\,.
\end{aligned}
\end{equation*}
Furthermore,
\begin{equation*}
\label{eq:equation to be continued}
\begin{split}
\Pi( Y\bar{\vq} - \bar{\mathbf{p}}) & \leq
(j -i)^m \Pi( Y\vq -\vp )  \leq \Pi_{+}(\vq)^{\frac{m\delta_0}{m+n+1 } }\, \Pi_{+}(\vq)^{-(1 + \delta_0) } \\
& \under{\equ{notethat}}\leq \Pi_{+}(\bar{\vq})^{-\frac{m+n+1 }{m+n+1 + n\delta_0}
{\frac{m\delta_0-(1 + \delta_0)(m+n+1) }{m+n+1 }}}
\\
&= \Pi_{+}(\bar{\vq})^{-(1+\frac{\delta_0}{m+n+1 + n\delta_0})} \,.
\end{split}
\end{equation*}
This, if we choose a positive $\delta$ not greater than $ \frac{\delta_0}{m+n+1+\delta_0}$
and assume, as we may, that $\Pi_{+}(\bar{\vq})^{\frac{\delta_0}{m+n+1+\delta_0} - \delta}$ is not less than $2$,
we obtain a solution $(\bar{\vp},  \bar{\mathbf{q}})$ to both
\equ{hom} and  (\ref{eq:equation with new condition}).
\end{proof}

\begin{lem}\label{lem: aux2} Suppose we are given $z_1,\dots,z_m \ge 0$, $r \ge 0$ and $C > 1$
such that \eq{given1}{z_i  < r\quad\text{for each }i = 1,\dots,m\,,}
 and  \eq{given2}{\prod_{i = 1}^m z_i  < r^m/C\,.}
Then there exist $C_1,\dots,C_m \ge 1$ such that
\eq{then1}{C = \prod_{i = 1}^m C_i\,, } and \eq{then2}{C_iz_i  \le r \quad\text{for each }i = 1,\dots,m\,.}
\end{lem}

\begin{proof} Without loss of generality assume that $z_m \le \dots \le z_1$.
Then define $C_0 = 1$ and inductively \eq{choice}{C_i = \min\big(\frac r{z_i}, \frac C{ \prod_{j = 0}^{i-1} C_j }\big)}
(here we use the convention $r/0 = \infty$).
The validity of  \equ{then2} is clear, and it follows from \equ{given1} that if  for some $i$ the first term
in the right hand side of   \equ{choice} is not less than the second one, the same
will happen for all the subsequent values of $i$. Also it follows from  \equ{given2} that 
a scenario under which $r/z_i < C/ \prod_{j = 0}^{i-1} C_j$ for all $i = 1,\dots,m$ is impossible.
Therefore for $i = m$ the minimum in  \equ{choice} is equal to the second term, implying \equ{then1}.\end{proof}

Now let us get back to the proof of the remaining part of Proposition  \ref{KM - rate of escape condition}(b). Suppose that  $Y$ is VWMA; in view of Lemma \ref{lem: aux1} we can assume that for some
$\delta >0$ there are infinitely many solutions to both 
\equ{hom} and (\ref{eq:equation with new condition}). Take an arbitrary positive $s < \frac1{m+n}$,
and for each solution $(\vp,\vq)$,
let $r = \Pi_{+}({\vq})^{-\delta s}$ and define $t_{m+1},\dots,t_n$ by 
$${|q_j|_{+} = r e^{t_{m+j}}\,.}$$
 Then $e^{-t_{m+j}}|q_j|\le e^{-t_{m+j}}|q_j|_+ = r$ and $\Pi_{+}({\vq}) = r^n e^t = \Pi_{+}({\vq})^{-\delta n s} e^t$,
 hence $ r =  e^{- \frac{\delta s}{1 + \delta n s}t}$ and $\Pi_{+}({\vq}) = e^{ \frac{1}{1 + \delta n s}t}$. Thus, denoting $\gamma =  \frac{\delta s}{1 + \delta n s}$, we have $e^{-t_{m+j}}|q_j|\le e^{-\gamma t}$ for $j = 1,\dots,n$. To finish the proof we need to find
 $t_1,
\dots,t_m \ge 0$ with $t = t_1 + \dots + t_m$ such that $e^{t_{i}}|Y_i\vq - p_i| \le e^{-\gamma
t }$ for each $i$; this would clearly imply the linear growth of  $g_{\mathcal{A}}  \Lambda_Y$.

For that, let us denote $z_i = |Y_i\vq - p_i|$ and $C = e^t$, and check  \equ{given1} and \equ{given2}:
in view of 
(\ref{eq:equation with new condition}), we have 
$$
z_i \le \Pi_{+}(\mathbf q)^{ - \delta/m  } = r^{1/ms} < r
$$
since $s < 1/m$, and also, in view of \equ{hom},
$$
\prod_{i = 1}^m z_i  =  \Pi( Y\vq -\vp ) \le \Pi_{+}(\mathbf q)^{ - (1 + \delta)  } =  e^{ -\frac{1 + \delta}{1 + \delta n s}t} = e^{-t}  e^{ -\frac{ \delta(1 - ns)}{1 + \delta n s}t}  =  e^{-t} r^{\frac{1-ns}{s}}\,,
$$
and the latter is not greater than $ r^m/C$ since $s < \frac1{m+n}$. Taking  $e^{t_i} = C_i $
where $C_1,\dots,C_m \ge 1$ are
as in Lemma \ref{lem: aux2}  finishes the proof. \end{proof}

\noindent {\bf Remark.} It easily follows from the continuity of the $G$-action on $\Omega$ that
whenever $\fs$ is a subset of $\mathcal{T}$ of bounded Hausdorff distance 
from $\mathcal{T}$ (that is, $\mathcal{T}$ is contained in the $r$-neighborhood of $\fs$ for some 
$r > 0$), $g_{\mathcal{T}}  \Lambda$ has linear growth if and only if so does $g_{\fs}  \Lambda$. 
In particular, without loss of generality we can take sets $\mathcal{T}$ to be countable, e.g.\ replace 
$\fa$ with the set of vectors in $\fa$ with integer coordinates. See some more explanations in the proof of \cite[Corollary 2.2]{KM}.

\medskip

The correspondence of Proposition \ref{KM - rate of escape condition}
will be instrumental in our deduction of the main results of this paper
from measure estimates on the space of lattices, following the method first introduced in \cite{KM}. 
Indeed, in view of the proposition, proving the extremality or strong extremality of  $F_*\nu$ is equivalent to showing that
for arbitrary  positive $\gamma$, $\nu$-almost every $\x$  is contained in at most finitely many sets
$\{\x : g_{\vt} \Lambda_{F(\x)}
\notin K_{e^{-\gamma t}}\big\}$, where  $\ft$ is either $\fr$ or $\fa$ and $\vt\in\ft$ has
integer coordinates. The latter will follow from the Borel-Cantelli Lemma and  estimates of type
\eq{measest}{
\nu\left(
\big\{\x\in B : g_{\vt} \Lambda_{F(\x)}
\notin K_\vre\big\} \right) \le \const \cdot \vre^\alpha \nu(B)\,,
} 
where $B\subset U$ is a ball and $\alpha > 0$.

Note that so far whenever the norm $\|\cdot\|$ on a finite-dimensional vector space
was used, in particular in the definition of the sets $K_\vre$, it was meant to be
the `maximum' norm. 
However replacing it by 
another norm would
result only in changes up to fixed multiplicative constants, and therefore Proposition \ref{KM - rate of escape condition} 
will remain true regardless of the norm used to define  $K_\vre$.
In what follows, for geometric reasons it will be convenient to describe sets
$K_\vre$ using Euclidean norm  $\|\cdot\|$ on $\R^{k}$ induced by the standard inner product
$\langle \cdot,\cdot \rangle$.

Note also that the geometry of $\Omega$ at infinity can be similarly described using other representations
of $G$, for example on higher exterior powers of $\R^k$. 
It will be convenient to  denote by $
\mathcal{W}_\ell$, where 
$1\le \ell \le k$, the set of elements $\vw = \vv_1\wedge \dots\wedge \vv_\ell$ of 
$\bigwedge^\ell(\Z^{k})$ where $\{\vv_1, \dots, \vv_\ell\in \Z^{k}\}$ can be completed
to a basis of $\Z^{k}$ (those are called {\sl primitive\/} $\ell$-tuples). In fact, up to a sign
elements of $\mathcal{W}_\ell$ can be identified with rational $\ell$-dimensional subspaces of 
$\R^{k}$, or, equivalently, with primitive subgroups of 
$\Z^{k}$ of rank $\ell$. We also let $$
\mathcal{W} \df \cup_{1\le \ell \le k}\mathcal{W}_\ell\subset \textstyle\bigwedge (\Z^{k})\,.$$
The  Euclidean norm and the inner product will  be extended  from
$\R^{k}$ to its exterior algebra; this way $\|\vw\|$ is equal to the covolume of the subgroup corresponding to $\vw$.
Then for $\varepsilon>0$  define
\begin{equation*}
\tilde K_{{\varepsilon}} \stackrel{\mathrm{def}}{=} \big\{g\Z^k\in\Omega \bigm| \Vert 
g\vw \Vert \geq {\varepsilon} \quad \forall\, \vw \in
\mathcal{W}\big\}\,.
\end{equation*}
Clearly $\tilde K_{{\varepsilon}} \subset  K_{{\varepsilon}}$; on the other hand it easily follows
from Minkowski's Lemma that for any positive $\vre$ one has
 $ K_{{\varepsilon}} \subset  \tilde K_{{c\varepsilon^{1/k}}}$ where $c > 0$ depends only on $k$. Therefore
 the following holds:

\begin{lem}\name{lem: higher lg} Given   $\mathcal{T}\subset \fa
$ and  $\Lambda\in\Omega$, $g_{\mathcal{T}}  \Lambda$ 
has linear growth
if and only if there exists $\gamma > 0$ such that
\eq{higher lg}{
g_{\T} \Lambda \notin \tilde K_{e^{-\gamma t}} \text{ 
 \ for an unbounded set of 
}
\vt\in\mathcal{T}
\,.
}
\end{lem}

\noindent {\bf Remark.} One can also use Proposition \ref{KM - rate of escape condition} 
for an alternative proof  of the multiplicative version of Khintchine's Transference Principle \cite{SW}, that is, 
the equivalence 
of $Y$ and $Y^T$ being VWMA. Indeed, let $\sigma$ be the linear
transformation of $\R^{k}$ induced by the permutation on the $k$ coordinates which
exchanges the group of the first $m$ of them with that of the last $n$, without reordering within groups,
and denote by $\varphi$ the automorphism of $G$ given by $\varphi(g) = \sigma\big((g^T)^{-1}\big)\sigma^{-1}$ for all $g\in G$. Then it is easy to see that $\varphi(g_\vt) = g_{\sigma(\vt)}$ and $\varphi(u_Y) = u_{-Y^T}$.  Since
$\varphi(\Gamma) = \Gamma$, the automorphism $\varphi$ induces a self-map of $\Omega$ which
we can also denote by $\sigma$; geometrically it can be interpreted as $\varphi(\Lambda) = \sigma(\Lambda^*)$ where $\Lambda^*$ is the lattice dual to $\Lambda$. Now the desired equivalence follows from an observation that
 $\varphi(K_\vre) \subset K_{c\vre^{k-1}}$ for all $\vre > 0$, where $c$ is a constant dependent only on $k$.

\section{Quantitative  nondivergence and its applications}
\name{nondiv}

During the last decade, starting from the paper 
\cite{KM}, quantitative nondivergence estimates for unipotent
trajectories on the space of lattices evolved into a powerful 
method
yielding measure estimates as in \equ{measest} for a certain  
broad class of measures $\nu$ and maps $F$. Recall that the  sets
in the left hand side of  \equ{measest} 
consist
of those $\x$ for which the lattice  $g_{\vt} \Lambda_{F(\x)}$ has 
a vector of length less than $\vre$. The crucial ingredient of the method
 is a way to keep track not just of length of vectors in that lattice, but 
 of covolumes of subgroups of arbitrary dimension. The following is our main estimate:

\begin{thm}[\cite{KLW}, Theorem 4.3]
\name{thm: friendly nondivergence}
Given $d,k\in\N$ and 
positive constants $\tilde C,D,\alpha$,
         there exists $C' = C'(d,k,\tilde C,\alpha,D) > 0$
with the following property.
Suppose  a measure  $\nu$  on $\R^d$ is  $D$-Federer on 
a ball $\til B$ centered at  $\supp\,\nu$, $0 < \rho \le 1$,
and $h$ is aa continuous map $\til B
\to G$ 
such that
for each $\vw \in \mathcal{W}$,
\begin{itemize}
\item[(i)]
the function   $\x\mapsto \|h(\x)\vw\|$  is $(\tilde C,\alpha)$-good on $\til B
$
with respect to
$\nu$,
           \end{itemize}
and
\begin{itemize} \item[(ii)]
\label{item: attain rho}
$  \|h(\x)\vw\|\geq \rho$ for some $\x\in\supp\,\nu \cap B$, where $B = 3^{-(k-1)}\til B$.
          \end{itemize}
Then for any $\,0<
\varepsilon
\leq \rho$,
$$
{\nu\big(\big\{\x \in B: 
{h}(\x)\Z^k
 \notin K_{\varepsilon}
\big\}\big)}\le C'
(\varepsilon/\rho)^{\alpha}{\nu(B)} \,.
$$
\end{thm}

This theorem has a long history, starting from Margulis' proof of non-divergence of unipotent flows
\cite{Mar:non-div},
and continuing with a series of papers by Dani \cite{Dani:inv, Dani:unip, Dani:unip2}. The way it appeared in  \cite{KLW} is 
essentially the same as in \cite{KM} but slightly generalized. The crucial step made in 
 \cite{KM} was the introduction of the requirement (condition labeled by (i) in the
 above theorem) that covolumes of subgroups
should give rise to \cag\ functions; this made it possible to significantly expand the applicability of the estimates. In particular, when 
 %
\eq{def h}{h(\x) = g_\vt u_{F(\x)}\,,}
where $\vt\in\fa$ and $F$ is a map from $U$ to $\mr$,
condition (i) will hold for balls $\tilde B$ centered at $\nu$-generic points
as long as $F$ and $\nu$ satisfy assumption (i) of Theorem  \ref{thm: strexgeneral}. 
To show this,
it will be helpful  to have explicit expressions for the coordinate functions of  $g_\vt u_{F(\x)}\vw$. 
Let us denote by $\{\ve_{1},\dots,\ve_{m},\vv_{1},\dots,\vv_{n}\}$ the standard basis of
$\R^{k}$. 
Then one has
\eq{actionbasis}{u_Y\ve_i = \ve_i
\quad\text{ and }\quad u_Y\vv_{j} = \vv_{j} + \sum_{i = 1}^m y_{i,j}\ve_i =\vv_{j} + \vy_j
\,,
}
where in the latter equality we have  identified
the columns  $\vy_1,\dots,\vy_n$ of $Y$ with elements of $E \df \Span(\ve_{1},\dots,\ve_{m})$ via the correspondence $\vy_j \leftrightarrow 
\sum_{i = 1}^m y_{i,j}\ve_i$. 

Now take
$I =
\{i_1,\dots,i_{r}\}\subset \{1,\dots,m\}$ and $J =
\{j_1,\dots,j_{s}\}\subset \{1,\dots,n\}$,
where $i_1 <\dots <i_{r}$ and $j_1<\dots<j_{s}$, and
consider
$\ve_{I} \df 
\ve_{i_1}\wedge\dots\wedge \ve_{i_{r}}$ and $\vv_{J} \df 
\vv_{j_1}\wedge\dots\wedge \vv_{j_{s}}$, with the convention $\ve_\vrn = \vv_\vrn = 1$. 
                              For any $1\le \ell \le k$, 
elements \eq{def basis}{\ve_{I}\wedge\vv_{J},\text{ where }I\subset \{1,\dots,m\} ,\ J\subset \{1,\dots,n\},\ |I| + |J| = \ell\,,} form a basis of $\bigwedge^{\ell}(\R^{k})$. Then one can write
\eq{action}{
\begin{split}
u_Y(\ve_{I}\wedge\vv_{J}) &= \ve_{I}\wedge(\vv_{j_1} + \sum_{i = 1}^m y_{i,j_1}\ve_i)\wedge\cdots\wedge
(\vv_{j_s} + \sum_{i = 1}^m y_{i,j_s}\ve_i) \\ &=
\sum_{L\subset
J}\sum_{\substack{K\subset  \{1,\dots,m\}\ssm I,\,
\\ |K| = |L|}}\pm y_{K,L}\ve_{I\cup K}\wedge\vv_{J\ssm L}
\,,\end{split}
}
where $y_{K,L}$ is defined as in \equ{def yij}, and the choice of sign in $\pm$ depends 
on $K$ and $L$. 

Now we can easily establish

\begin{lem}\label{lem: good}  Let $d,k\in\N$ and $C,\alpha > 0$, and suppose $\til B$ is a ball in $\R^d$, $\nu$ is a measure on  $\til B$,   and $F:\til B\to\mr$ is a continuous map such that $(\vd \circ F,\nu)$ is 
\cag\ on  $\til B$.
Then functions   $\x\mapsto \|g_\vt u_{F(\x)}\vw\|$  are $(N^{\alpha/2} C,\alpha)$-good on $\til B
$
with respect to
$\nu$  for any $\vt\in\fa$ and $\vw\in \mathcal{W}$,  where $N$ 
is as in \equ{def n}.
\end{lem}

\begin{proof} Take  $\vw\in\bigwedge^{\ell}(\R^{k})$. In view of \equ{action},
each coordinate  of $u_{F(\cdot)}\vw$ with respect to the
basis \equ{def basis} is a linear combination of functions $F(\cdot)_{K,L}$
for various $ K\subset  \{1,\dots,m\}$ and  $ L\subset  \{1,\dots,n\}$
with $|K| = |L|$, that is, of $1$ and components of 
$\vd \circ F$.
The same can be said about coordinates of $g_\vt u_{F(\cdot)}\vw$; in fact,  the
basis 
 \equ{def basis} consists of eigenvectors for  $g_\vt$. It remains to apply a well-known
and elementary property, see e.g.\ \cite[Lemma 4.1]{KLW}, that  whenever  $f_1,\dots,f_N$
are $(C,\alpha)$-good on a set
$U$ with respect to a measure
$\nu$, the function $(f_1^2 + \dots +f_N^2)^{1/2}$ is $(N^{\alpha/2}C,\alpha)$-good
on
$U$ with respect to
$\nu$.
\end{proof}

Consequently, whenever $F$ and $\nu$ satisfy assumption (i) of Theorem  \ref{thm: strexgeneral}
(in particular, if $F$ is real analytic and $\nu$ is Lebesgue measure), 
for $\nu$-almost all $\x$ it is possible to choose a ball $\til B$ centered at $\x$ and 
$\til C,\alpha > 0$ such that $h(\cdot)$ as in \equ{def h} satisfies condition (i) of Theorem \ref{thm: friendly nondivergence}.
Our attention will be thus centered on lower bounds for $\|g_\vt u_{F(\cdot)}\vw\|_{\nu,B}$; 
indeed, a bound uniform in $\vw$ and $\vt$ would make it possible to apply Theorem  \ref{thm: friendly nondivergence} and establish \equ{measest}. Moreover, generalizing a result from \cite{gafa}
it is possible to write down a  condition equivalent to the statement
\eq{nolg}{g_{\mathcal{T}}
\Lambda_{F(\x)}\text{ 
 has linear growth for $\nu$-almost no }\vx}
within the class of Federer measures and good pairs. 

\begin{thm}\label{thm: criterion}   Let     an open subset  $U$  of $\R^d$, a  continuous map
  $F: U\to\mr
$
and 
 a  Federer measure   $\nu$ on
$U$ be such that the pair $(\vd\circ F,\nu)$ is good. 
 Also let $\ft$ be an unbounded subset  of $\fa$. Then \equ{nolg} holds if and only if
 for any  ball
$B\subset U$ with $\nu(B) > 0$ 
and any $\beta
> 0$  there exists $T  > 0$ such that 
\eq{condition}{\|g_\vt u_{F(\cdot)}\vw\|_{\nu,B} \geq e^{-\beta t}\quad\forall\,\vw\in\fw \text{ and any }
\,\vt\in\ft\text{ with }t\ge T\,.}
\end{thm}

\begin{proof} Let us start with the `if' part. Take an arbitrary positive $\gamma$. Since  $\nu$ is Federer, $(\vd\circ F,\nu)$ is good and in view of 
Lemma \ref{lem: good},  for $\nu$-almost every $\x_0\in U$ there exists a ball $\til B$ 
centered at $\x_0$ and constants $\til C,\alpha,D$ such that all the functions $\x\mapsto \|g_\vt u_{F(\x)}\vw\|$  are $(\til C,\alpha)$-good on $\til B
$
with respect to
$\nu$, and $\nu$ is $D$-Federer on $\til B$. Then take $B = 3^{-(k-1)}\til B$, choose an arbitrary $0 < \beta < \gamma$ and  $T$ 
such that \equ{condition} holds. This will enforce condition (ii) of  Theorem  \ref{thm: friendly nondivergence} with $\rho = e^{-\beta t}$ and $h$ as in \equ{def h} with $t\ge T$.
Applying Theorem  \ref{thm: friendly nondivergence} with $\vre = e^{-\gamma t}$
will yield
$$
{\nu\big(\big\{\x \in B: g_\vt \Lambda_{F(\x)}\ \notin K_{e^{-\gamma t}}
\big\}\big)}\le C'
(e^{-(\gamma - \beta) t})^{\alpha}{\nu(B)} \,.
$$
Now choose a countable subset  $\fs$ of $ \mathcal{T}$ with finite Hausdorff
 distance from $ \mathcal{T}$ and such that $\inf_{\vt_1,\vt_2\in \fs,
 \vt_1\ne\vt_2}\| \vt_1-\vt_2\| > 0$.
The above estimate implies that
$$
\sum_{\T\in\fs}\nu\left(
\big\{\x\in B : g_{\vt} \Lambda_{F(\x)}
\notin K_{e^{-\gamma t}}\big\} \right) < \infty\,.$$
Applying   the Borel-Cantelli Lemma, one concludes that for $\nu$-a.e.\  $\x\in B$ one has
$$g_{\vt} \Lambda_{F(\x)}
\in K_{e^{-\gamma t}}$$ for all
but finitely many $\T \in \fs$, which, in view of the remark before Proposition \ref{KM - rate of escape condition} and since $\gamma$ could be chosen arbitrary small, implies 
 \equ{nolg}.

As for the converse, suppose that there exists a ball $B\subset U$
 with $\nu(B) > 0$ and $\beta > 0$ such that
 for an unbounded set of $\vt\in\ft$ one has
$$\|g_\vt u_{F(\cdot)}\vw\|_{\nu,B} < e^{-\beta t}$$ for some $\vw\in\fw$ (dependent on $\vt$).
This means that for any $\x\in  B\cap \supp\,\nu$ and 
for each $\vt$ as above, $g_\ft \Lambda_{F(\x)}$ is not in $\tilde K_{e^{-\beta t}}$,
This, in view of Lemma \ref{lem: higher lg}, implies that  $g_\ft \Lambda_{F(\x)}$ has linear growth 
for all $\x$ in
$ B\cap \supp\,\nu$.
\end{proof}
   
In particular, in view of Proposition \ref{KM - rate of escape condition}, 
for $\ft = \fr$ or $\fa$ we get criteria for extremality and strong extremality of 
$F_*\nu$  within the class of good pairs. Note that here we see a dichotomy
between a certain property happening either for almost no points or for all points in  some 
nonempty open ball. This is typical
for this class of problems, see \cite{gafa, dima tams, dichotomy, yuqing}.

\section{Proof of Theorem \ref{thm: strexgeneral} 
}
\name{exp}

In general, checking a conditions like \equ{condition} 
seems to be a complicated task; 
the full strength of the vector case 
  ($n=1$) of Theorem \ref{thm: criterion} has been utilized in  \cite{gafa}, see also \cite{dima tams, yuqing}.
However, we will show that the nonplanarity assumption of Theorems \ref{thm: strexgeneral} implies a stronger
property, namely $e^{-\beta t}$ in the right hand side of  \equ{condition} can be replaced by a 
positive constant dependent only on $B$. To establish such lower bounds,
we are going to  look closely at projections of `curves'
$\{u_{F(\x)}\vw\}$ in $\bigwedge(\R^k)$ onto subspaces expanded by the 
$g_\vt$-action. Namely, for a fixed $\vt$ let us denote by $E^+_\vt$ the span of all 
the eigenvectors of $g_\vt$ in  $\bigwedge(\R^k)$ with eigenvalues greater or equal to one (in other words, those
which are not contracted by the $g_\vt$-action). It is easy to see that $E^+_\vt$ is spanned by 
elements $\ve_I \wedge \vv_J$ where 
$I \subset
\{1,\dots,m\}$ and $J \subset
\{1,\dots,n\}$ are such that $$\sum_{i\in I}t_{i} \ge \sum_{j \in J}t_{m+j}\,.$$
Also let 
 $\pi^+_\vt$ be the orthogonal projection onto $E^+_\vt$. 
%
 As a straightforward application of Theorem \ref{thm: criterion}, we have

\begin{cor}\name{cor: uniform lower}   Let   
  $F: U\to\mr
$,
$\nu$ 
 and $\ft$ be as in  Theorem \ref{thm: criterion}. 
 Suppose that  for any  ball
$B\subset U$ with $\nu(B) > 0$ 
one has
 \eq{cor condition}{\inf_{ \vw\in\fw,\, \vt\in\ft} \|\pi^+_\vt u_{F(\cdot)}\vw\|_{\nu,B} > 0
 \,.}
Then \equ{nolg} holds.
\end{cor}

\begin{proof} By the definition of the map $\pi^+_\vt$, for any $\vw\in \bigwedge(\R^k)$ one has
$$
  \|g_\vt \vw\| \ge \|\pi^+_\vt g_\vt \vw\| =  \|g_\vt \pi^+_\vt \vw\|\ge  \| \pi^+_\vt \vw\|\,,
  $$
hence a uniform lower bound, say $c$, on   $\| \pi^+_\vt  u_{F(\x)}\vw\|$  implies a similar bound on
  $\|g_\vt u_{F(\x)}\vw\|$. Thus   \equ{condition} will hold as long as $e^{- \beta T} \le c$.
   \end{proof}

Our strategy for checking extremality or strong extremality will be to derive estimates
of type  \equ{cor condition} from the nonplanarity assumptions, in particular from those of Theorem 
 \ref{thm: strexgeneral}. However before proceeding let us exhibit a partial converse to the above corollary:

\begin{cor}\name{cor: uniform zero}  
 Let  \amr\ and let $\ft$  be an unbounded subset  of $\fa$.      Suppose that  there exist  $\vt_0\in\ft$  
 and $\vw\in\fw$ such that:
 \begin{itemize}
 \item[(a)] $\fs\df\{c\vt_0  :  c > 0\} \cap \ft$ is unbounded; and
 \item[(b)]  $ \pi^+_{\vt_0} u_{Y}\vw = 0$.
  \end{itemize}
 Then  $g_\ft \Lambda_{Y}$ has linear growth.
\end{cor}

\begin{proof} From (a) and (b) it follows that 
$u_{Y}\vw$ belongs to the
orthogonal complement of $E_\vt^+$ whenever $\vt\in\fs$ (clearly the spaces $E_\vt^+$ do not change if 
$\vt$ is replaced by a proportional vector). Hence it is exponentially contracted by the $g_\vt$-action, that is,
  for some 
$\beta > 0$ and all $\vt\in\fs$  one can write
$$
\|g_\vt u_{Y}\vw\| \le e^{-\beta t}\|u_{Y}\vw\| \le C e^{-\beta t}\,,
$$
where $C$ is a constant depending on $\vw$ and $Y$.  Consequently  \equ{higher lg} is satisfied with $\Lambda = \Lambda_{Y}$, 
and Lemma \ref{lem: higher lg} readily implies the linear growth of  $g_\ft \Lambda_{Y}$. \end{proof}

In particular, whenever conditions (a) and (b) above are satisfied for some $\vt_0\in\ft$ and $Y$ of the form $F(\x)$
for all $\x\in\supp\,\nu\cap B$, where $B\subset U$ is a ball of positive measure (that is, the infimum in the left hand side of \equ{cor condition} is equal to zero and {\it is attained\/}), it follows  that 
  \equ{nolg} does not hold, and, moreover, $g_\ft \Lambda_{F(\x)}$ has linear growth for all $\x\in B\cap \supp\,\nu$. We will explore this when it comes to discussing specific examples at the end of the paper.

\medskip

Now let us get back to Corollary \ref{cor: uniform lower}  and its applications.
The next observation immediately follows from 
 the compactness of spheres in finite-dimensional spaces:
 
\begin{lem}\label{lem: nonpl}   Let   $\nu$  be a measure on $\R^d$
and
 $\vf = (f_1,\dots,f_N)$ a map $U\to \R^N$, where $U\subset \R^d$ is open with $\nu(U) > 0$.  
Then $(\vf,\nu)$ is nonplanar if and only if for any ball $B\subset U$ with $\nu(B) > 0$
there exists $c > 0$ 
such that
$$\|a_0 + \sum_{i = 1}^Na_if_i \|_{\nu,B} \ge c\text{ for any }a_0,a_1,\dots,a_N 
\text{ with }\max |a_i| \ge 1,.
$$  \end{lem}

Since it is assumed in Theorem 
 \ref{thm: strexgeneral} that $(\vd\circ F,\nu)$ is nonplanar, in view of the above lemma and corollary to check  \equ{cor condition}  it would suffice to bound $ \|\pi^+_\vt u_{F(\cdot)}\vw\|$ from below by the absolute value of a  linear combination  of $1$ coordinates
 of $\vd\circ F$ with big enough coefficients.
 
 \medskip
 
 Note that the spaces $E^+_\vt$ may be different for different  $\vt$ (although, as was mentioned above,
 $E^+_\vt = E^+_{\vt'}$ if $\vt$ and $\vt'$ are proportional). However, it turns out that in the set-up
 of Theorem  \ref{thm: strexgeneral} one can work with the intersection of all those spaces:
 $$
 E^+ \df \cap_{\vt\in\fa}E^+_\vt 
 $$
 consisting of elements which are not contracted by $g_\vt$ for all $\vt\in\fa$. It is easy to see that $
 E^+ $ is spanned by   \eq{basisplus}{\big\{\ve_I, \ve_{\{1,\dots,m\}}\wedge \vv_J: I \subset
\{1,\dots,m\},J \subset
\{1,\dots,n\}\big\}\,.}
The next proposition
explains that for any $\vw\in\fw
$ it is possible to find an element of $E^+$ on which
the  `curves'
$\{u_{F(\x)}\vw\}$ project nontrivially.


\begin{prop}\name{prop: nonzero}  For any $\vw\in\fw$ 
 it is possible to choose
an element $\vw_0$ of the basis  \equ{basisplus} of  $E^+$ such that 
the function $Y\mapsto
 \langle  u_{Y}\vw, \vw_0\rangle$ is a nontrivial integer linear combination 
of $1$ and components of $\vd(Y)$.
\end{prop}


\begin{proof} Denote by $\pi^+$ the orthogonal projection onto $E^+$.  We are going to
use  \equ{action} to explicitly write down the coordinates  
$\pi^+u_Y\vw$ with respect to the basis  \equ{basisplus} for any
$\vw\in\fw_\ell$, that is,
\eq{w}{\vw = \sum_{I,J,\,|I| + |J| = \ell}a_{I,J}\ve_{I}\wedge\vv_{J}
\,.}
Consider two cases. \medskip 

{\bf Case 1.} If 
$ \ell \le m$, 
using \equ{action} one can see that
$${\pi^+u_Y(\ve_{I}\wedge\vv_{J})= \ve_{I}\wedge \sum_{K\subset  \{1,\dots,m\}\ssm I,\, |K| = |J|}\pm y_{K,J} \ve_K\,.}
$$
Note that $|I|$ can take values between $\max(0,\ell-n)$ and $\ell$; equivalently, $|J| = \ell - |I|$ ranges between $0$ and $\ell - \max(0,\ell-n) = \min(\ell,n)$.
Thus
$$
\pi^+u_Y\vw  = \sum_{\substack{I\subset \{1,\dots,m\}\\\ \max(0,\ell-n)\le |I|\le \ell}}
\ve_{I}\wedge \sum_{\substack{J\subset  \{1,\dots,n\}\\\ |J| = \ell - |I|}} a_{I,J}\sum_{\substack{K\subset  \{1,\dots,m\}\ssm I \\ |K| = |J|}}\pm y_{K,J} \ve_K\,.
$$
Rearranging terms and substituting $L = I\cup K$, we get
$$
\pi^+u_Y\vw = \sum_{\substack{L\subset \{1,\dots,m\}\\\  |L| = \ell}}
\left( \sum_{\substack{K\subset  L\\ 0 \le |K| \le \min(\ell,n)\le |I|\le \ell}}\  \sum_{\substack{J\subset  \{1,\dots,n\}\\\ |J| = |K|}}\pm a_{L\ssm K,J} y_{K,J}\right) \ve_L\,.
$$
\ignore{
Now recall that we are also given  $\vt\in\fa$. Without loss of generality assume that the components of $\vt$ are ordered so that
$t_1 = \max_{i = 1,\dots,m} t_i$.
This implies 
that $t_1\geq t/m$, and, therefore
\eq{expansion1}{\Vert g_{\T}\ve_{L}\Vert\geq e^{t_1}\Vert\ve_{L}\Vert\geq
e^{t_{1}}\geq e^{t/m} } 
whenever $L\subset \{1,\dots,m\}$ contains
 $1$.
We also know that $\vw\notin E^+$, thus
  $a_{I,J} \ne 0$ for some $J\ne\varnothing$. 
  }
  Recall that the coefficients in the expansion \equ{w} are integer and at least one
of them, say $a_{I,J}$, is nonzero. Take any $K\subset \{1,\dots,m\}\ssm I$
  with $|K|= |J|$ and denote $L\df I\cup K$. Then 
  $$
  \langle  u_{Y}\vw, \ve_L \rangle = \sum_{\substack{K\subset  L\\ 0 \le |K| \le \min(\ell,n)\le |I|\le \ell}}\  \sum_{\substack{J\subset  \{1,\dots,n\}\\\ |J| = |K|}}\pm a_{L\ssm K,J} y_{K,J}
  $$
  will be a nontrivial (since  $a_{I,J}$ is one of the coefficients) integer  linear combination 
of $1$ and components of  $\vd(Y)$.


\medskip 

{\bf Case 2.} If $\ell \ge m$, we get
\begin{equation*}
\begin{aligned}
\pi^+u_Y(\ve_{I}\wedge\vv_{J}) &= \ve_{I}\wedge \left(\sum_{K\subset  J,\, |K| = m - |I|}\pm y_{\{1,\dots,m\}\ssm I,K} \ve_{\{1,\dots,m\}\ssm I} \wedge \vv_{J\ssm K}\right)\\ &= \ve_{\{1,\dots,m\}}\wedge\left(\sum_{K\subset  J,\, |K| = m - |I|}\pm y_{\{1,\dots,m\}\ssm I,K}\wedge \vv_{J\ssm K}\right)\,.
\end{aligned}
\end{equation*}
Note that this time we must have $\max(0,\ell-n)\le |I|\le m$, or, equivalently, $
\{1,\dots,m\} \ssm I| \le m - \max(0,\ell-n) = \min(m,k-\ell)$.
Therefore:
$$
\pi^+u_Y\vw = \ve_{\{1,\dots,m\}}\wedge\sum_{\substack{I\subset \{1,\dots,m\}\\  |I| \ge \max(0,\ell-n)}}\ 
 \sum_{\substack{J\subset \{1,\dots,n\}\\ |J| = \ell - |I|}} a_{I,J}\sum_{\substack{K\subset J \\ |K| = m - |I|}}
\pm y_{\{1,\dots,m\}\ssm I,K}\vv_{J\ssm K}\,.
$$
Rearranging terms, substituting $L = J\ssm K$ and replacing $I$ with $\{1,\dots,m\}\ssm I$, we get
$$
\pi^+u_Y\vw \ 
= \sum_{\substack{L\subset \{1,\dots,n\}\\  |L| = \ell-m}} 
\left( \sum_{\substack{I\subset  \{1,\dots,m\}\\  |I| \le   \min(m,k-\ell)}}\  \sum_{\substack{K\subset  \{1,\dots,n\}\ssm L\\ |K| =  |I|}}
\pm a_{\{1,\dots,m\}\ssm I,K\cup L} y_{ I,K}\right)\ve_{\{1,\dots,m\}}\wedge\vv_L\,.
$$
 \ignore{
 This time we can without loss of generality assume that  $t_{m+1} = \max_{j = 1,\dots,n} t_{m+j}$, hence
 $t_{m+1}\geq t/n$ and \eq{expansion2}{\Vert g_{\T}(\ve_{ \{1,\dots,m\}}\wedge\vv_{L})\Vert\geq  e^{t_{m+1}}\Vert \ve_{ \{1,\dots,m\}}\wedge\vv_{L}\Vert\geq e^{t/n}}
 whenever $L\subset \{1,\dots,n\}$ with 
 $1\notin L$. }
 Now let  $ a_{\{1,\dots,m\}\ssm I,J} $ be a nonzero coefficient.
Then one can take any $K\subset   J$ with $|K|= |I|$ 
and conclude 
 that
  $$
  \langle u_{Y}\vw, \ve_{\{1,\dots,m\}}\wedge\vv_L \rangle =  \sum_{\substack{I\subset  \{1,\dots,m\}\\  |I| \le   \min(m,k-\ell)}}\  \sum_{\substack{K\subset  \{1,\dots,n\}\ssm L\\ |K| =  |I|}}
\pm a_{\{1,\dots,m\}\ssm I,K\cup L} y_{ I,K}
  $$
  is a nontrivial (since  $a_{\{1,\dots,m\}\ssm I,J}$ is one of the coefficients) integer  linear combination 
of $1$ and the components of  $\vd(Y)$.
This finishes the proof
of the proposition.
\end{proof}

We remark that in the case $m = 1$ or $n = 1$ all the spaces $E^+_\vt$, $\vt\in\fa$, coincide with
each other and with $E^+$; 
in that case  in \cite{KM} and \cite{KLW}   a simplified form of the above computation 
was used to prove \cite[Theorem 5.4]{KM} and \cite[Theorem 3.3]{KLW}  respectively.


Finally we can complete the 

\begin{proof}[Proof of Theorem 
 \ref{thm: strexgeneral}]  Recall that it suffices to check that the
 assumption  of Corollary \ref{cor: uniform lower} are satisfied. Take $B\subset U$ with $\nu(B) > 0$,
 and write
 $$\|\pi_\vt^+ u_{F(\cdot)}\vw\| \ge | \langle\pi_\vt^+ u_{F(\cdot)} \vw,\vw_0\rangle|  =  | \langle u_{F(\cdot)} \vw,\vw_0\rangle| \cdot\|\pi_\vt^+ \vw_0\|\ge  | \langle u_{F(\cdot)} \vw,\vw_0\rangle| \,,
 $$
where $\vw_0$  is as in Proposition \ref{prop: nonzero}, so that  $ \langle u_{F(\cdot)} \vw,\vw_0\rangle$
is a nontrivial integer linear combination of $1$ and the components of $\vd\circ F$.
Therefore Lemma \ref{lem: nonpl} and the nonplanarity of $(\vd\circ F,\nu)$ imply 
\equ{cor condition}.
 In view of Corollary \ref{cor: uniform lower}, 
 Theorem \ref{thm: criterion}  and Proposition \ref{KM - rate of escape condition}(b), $F_*\nu$ is strongly extremal.
 \end{proof}

\section{Consequences of Theorem \ref{thm: strexgeneral}}
\name{indepvar}
Our goal in this section is to construct examples of pairs $(F,\nu)$ such that
the assumptions of Theorem \ref{thm: strexgeneral} 
are satisfied. As mentioned in \S\ref{statements}, whenever $\vf: U\to \R^N$ is a nondegenerate smooth map, the pair $(\vf,\lambda)$ is  
good and
nonplanar   (see \cite[Proposition 3.4]{KM}). 
Since Lebesgue measure is  Federer, Theorem  \ref{thm: strexgeneral} as a special case implies

\begin{cor}\name{cor: nondeg} Let  $F:U\to \mr$ be a differentiable map such that $\vd\circ F$ is nondegenerate.  Then $F_*\lambda$
is strongly extremal.
\end{cor}

Specific examples include
\eq{examples new}{
\begin{aligned}x \mapsto \begin{pmatrix} x & x^2\\ x^3 & x^5\end{pmatrix}\,,\ \  \text{ or } \ \ x \mapsto  \begin{pmatrix} x & x^2 & x^3  \\ x^4 & x^6 & x^8\end{pmatrix}\,,\end{aligned}} 
where $x\in\R$. 
More generally, here is a definition introduced in \cite{KLW}: given 
 $C$, $\alpha > 0$ and  an open subset $U$ of $\R^d$,   say that
$\nu$ is 
{\em absolutely decaying\/} if for  $\nu$-a.e.\
$\x \in \R^n$  there exist a neighborhood $U$ of $\x$ and $C,\alpha >
0$ such that  for 
any non-empty open ball $B \subset U$ centered at $\supp\,\nu$, 
any affine hyperplane $\mathcal{L} \subset
\R^n$  and any $\varepsilon >0$
one has
\begin{equation*}
\label{eq: defn abs decaying}
{\nu \left( B \cap \mathcal{L}^{(\varepsilon)} \right) }
\le C
\left(
\frac{\varepsilon}{r}
\right)^{\alpha}{\nu(B)}\,,
\end{equation*}
where $r$ is the radius of $B$ and $\mathcal{L}^{(\varepsilon)}$
is the $\vre$-neighborhood of $\mathcal{L}$. 
%
%
 Following a terminology suggested in \cite{PV}, say that
$\nu$ is {\sl absolutely friendly\/} if it is Federer and absolutely decaying. The following was essentially 
proved in \cite{KLW} (see \cite[Theorem 2.1(b) and \S7]{KLW}): suppose that  $\nu$ is an absolutely friendly measure, $\ell\in\N$, and $\,\vf$ is a 
$C^{\ell+1}$  map which  is $\ell$-nondegenerate at
$\nu$-a.e.\ point; then $(\vf,\nu)$  is good. Since the nonplanarity of 
$(\vf,\nu)$ is immediate from the nondegeneracy condition,  the following is 
also a special case of Theorem \ref{thm: strexgeneral}:

\begin{cor}\name{cor: af}
Let $\nu$ be an absolutely friendly measure on $\R^d$, $U$ an open subset of
$\R^d$, $\ell\in\N$, and   $F:U\to\mr$ a $C^{\ell+1}$  map such that $\vd \circ F$ is 
$\ell$-nondegenerate at
$\nu$-almost every point.  Then $F_*\nu$
is strongly extremal.
\end{cor}

Numerous examples of absolutely friendly measures have been constructed  in 
\cite{KLW, bad, Urbanski, Urbanski-Str}. In particular, limit measures of finite 
irreducible systems of contracting similarities \cite[\S8]{KLW}
(or, more generally, self-conformal contractions, \cite{Urbanski}) 
satisfying the
open set condition are absolutely friendly. Thus, if $\nu$ is, say, the natural measure
on the Cantor set in $\R$ and $F$ is one of the maps of the form \equ{examples new},
the pushforward of $\nu$ by $F$ is strongly extremal.

\medskip

In general checking the nondegeneracy of $\vd \circ F$ may be a complicated task.
However, in the important special case 
when 
the rows (or 
columns) of $F$ are functions of independent variables, 
the assumptions of Theorem \ref{thm: strexgeneral} turn out to be easier to check.
Namely, the following is true:

\begin{thm}\name{thm: indepvargeneral} For every $i = 1,\dots,m$, let $\vf_i$ be a continuous
 map
from an
 open subset $U_i$ of $\br^{d_i}$ to $\br^n$, and  let $\nu_i$ be a Federer measure on
$\br^{d_i}$  such that for each $i$, the pair $(\vf_i,\nu_i)$ is good and nonplanar. 
Define $F$ by \equ{f} and let $\nu = \nu_1\times\dots\times \nu_m$. Then  {\rm (a)} $\nu$ is Federer,
and  $(\vd \circ F,\nu)$ is 
 {\rm (b)}  good and
 {\rm (c)} nonplanar.
\end{thm}

In view of the discussion preceding Corollary \ref{cor: nondeg}, the above theorem includes Proposition  \ref{prop: indepvar} as a special case;
hence its proof also establishes Theorem \ref{thm: strex}.

\begin{proof} The fact that the  product of Federer measures is Federer is straightforward,
see e.g.\ \cite[Theorem 2.4]{KLW}. For parts  (b) and (c) we will use  induction on $m$.
The case $m = 1$ is obvious since in that case $\vd \circ F$ is the same as $F$.
The induction step is based on the following elementary 
observation: given  \amr\ with $m > 1$, any linear combination  of   components 
of $\vd(Y)$  and $1$, that is,
\eq{lincomb1}{
\sum_{\substack{I\subset  \{1,\dots,m\},\  J\subset  \{1,\dots,n\}\\  0  \le  |I| = |J| \le \min(m,n)}}a_{I,J}y_{I,J}
}
can be rewritten as 
\eq{lincomb2}{
\sum_{\substack{I\subset  \{2,\dots,m\},\  J\subset  \{1,\dots,n\}\\  0  \le  |I| = |J| \le \min(m-1,n)}}\left( a_{I,J} + \sum_{j \notin J} a_{I \cup\{1\},J\cup\{j\}}y_{1,j}\right)y_{I,J}\,,
}
where the choice of signs in $\pm$ depends on  $j$ and $J$. 

Let us first establish the nonplanarity of $(\vd \circ F,\nu)$. 
Denote $\x' = (\x_2,\dots,\x_m)$, $\nu' = \nu_2\times\dots\times \nu_m$,  and let
$$
F': U_2\times\dots\times U_m\to\mr,\quad \vx' \mapsto \begin{pmatrix} \vf_2(\vx_2)\\
\vdots\\
\vf_m(\vx_m)\end{pmatrix}
$$
Assume that the statement is true for $m-1$ in place of $m$, which in particular 
implies that the pair $(\vd\circ F',\nu') $ is nonplanar.
Take a ball $B\subset \R^{d_1 + \dots + d_m}$ with $\nu(B) > 0$. 
Choose coefficients $a_{I,J}\in\R$, where $I\subset  \{1,\dots,m\}$, $J\subset  \{1,\dots,n\}$, 
 $0  \le  |I| = |J| \le \min(m,n)$, such that one of them has absolute value at least $ 1$, and denote
$$
\varphi(\x) = \sum_{I,J}a_{I,J}f(\x) _{I,J}
$$ 
Then, using the equivalence of \equ{lincomb1} and \equ{lincomb2}, one can write
\eq{phi}{
\varphi(\x) =  
\sum_{\substack{I\subset  \{2,\dots,m\},\  J\subset  \{1,\dots,n\}\\  0  \le  |I| = |J| \le \min(m-1,n)}}\left( a_{I,J} + \sum_{j \notin J} a_{I \cup\{1\},J\cup\{j\}}f_{1,j}(\x_1)\right)f(\x')_{I,J}\,.
}
Since $\max |a_{I,J}| \ge 1$, one can choose $I\subset  \{2,\dots,m\}$ and $  J\subset  \{1,\dots,n\}$
such that the absolute value of some coefficient  in the expression 
$$a_{I,J} + \sum_{j \notin J} a_{I \cup\{1\},J\cup\{j\}}f_{1,j}$$ is at least $ 1$.
Since $(\vf_1,\nu_1)$ is nonplanar, Lemma \ref{lem: nonpl} implies that there exists $\tilde\x_1\in B_1\cap \supp\,\nu_1$
and $c_1 > 0$ such that $|a_{I,J} + \sum_{j \notin J} a_{I \cup\{1\},J\cup\{j\}}f_{1,j}(\x_1)|\ge c_1$. 
Fixing $\vx_1 = \tilde\vx_1$, we infer that at least one of the functions $f(\cdot)_{I,J}$ in the
linear combination \equ{phi} has a coefficient of absolute value at least $c_1$. From  the nonplanarity of $(\vd\circ F',\nu')$ we can then    deduce the existence of $c > 0$ and $\vx'$ such that $\vx = (\tilde\vx_1,\vx')\in B\cap \supp\,\nu$ and $|\varphi(\x)| > c$. This, again in view of Lemma  \ref{lem: nonpl},  shows the nonplanarity of $(\vd\circ F,\nu)$.

\medskip 
The proof of part (c) goes along similar lines and is based on the following 

\begin{lem}[\cite{KT}, Lemma 2.2]\label{lem: prod good}   Let   metric spaces $X,Y$ with measures $\mu,\nu$
be given. Suppose $\varphi$ is a continuous function  on $U\times V$,
where $U\subset X$ and $V\subset Y$ are open subsets, and
suppose $C,D, \alpha, \beta$
are positive constants such that
$$
\aligned
\text{for  all $y\in V\cap\,\supp\,\nu$, the function }x\mapsto \varphi&(x,y)\\
\text{ is
$(C,\alpha)$-good on $U$ with respect to $\mu$}\,,
\endaligned
$$
and
$$
\aligned
\text{for  all $x\in U\cap\,\supp\,\mu$, the function }y\mapsto \varphi&(x,y)\\
\text{  is
$(D,\beta)$-good on $V$  with respect to $\nu$}
\,.\endaligned
$$
Then $\varphi$ is $(E,\gamma)$-good on $U\times V$  with respect to $\mu\times \nu$,
         where $E$ and $\gamma$ can be explicitly expressed in terms of $\alpha, \beta, C, D$.
\end{lem}

It is given that $\nu_1$-a.e.\ point of $\R^{d_1}$ 
has a neighborhood $U_1$ such that $(\vf_1,\nu_1)$ is $(C_1,\alpha_1)$-good on $U_1$
for some $C_1,\alpha_1 > 0$. 
From the induction assumption  it follows that $\nu'$-a.e.\ point of 
$\R^{d_2 + \dots + d_m}$
has a neighborhood $U'$ such that $(\vd\circ F',\nu')$ is $(C',\alpha')$-good on $U'$
for some $C',\alpha' > 0$. 
Taking $U = U_1$, $V = U'$
and  $\varphi$ as in \equ{phi}, one sees that the assumptions of the above 
lemma are satisfied, and therefore for $\nu$-a.e.\ $(\x_1,\x')$ there exists a 
 neighborhood $U$ of $ (\x_1,\x')$ and $C,\alpha > 0$ such that $(\vd\circ F,\nu)$ is $(C,\alpha)$-good on $U$. This finishes the proof of Theorem \ref{thm: indepvargeneral}.
 \end{proof}

\section{Low-dimensional examples}
\name{moreexamples}

It is not hard to guess, looking at the information used in Corollary \ref{cor: uniform lower},
that 
it might be possible to weaken the nonplanarity assumption  of Theorem \ref{thm: strexgeneral} by requiring only some, and not all, linear combimations of components of 
$\vd\circ F$  to be nonzero.
In this section we consider some low-dimensional special cases 
and exhibit conditions
sufficient for strong extremality and extremality of $F_*\nu$ which are weaker
than the ones required by Theorem \ref{thm: strexgeneral}, thus  generating
 new examples of 
extremal and strongly extremal measures.

With some abuse of notation, let us
 introduce the following definition: say that a pair $(F,\nu)$, where
$F: U\to \mr$ and $\nu$ is a measure on $U$, is {\sl nonplanar\/} if for any ball $B\subset U$
with $\nu(B) > 0$ and any   nonzero $\vv\in\R^n$, the 
restriction of the map
$\vx\mapsto F(\vx) \vv$ to $B\,\cap \,\supp\,\nu$
is nonconstant. Clearly it coincides with the definition
of nonplanarity if $m = 1$, and clearly  $(F,\nu)$ is  row-nonplanar if
$F$ has a row $\vf$ such that $(\vf,\nu)$ is  nonplanar (but converse is not true).

\medskip

For the first result of this section, let us take $m = n = 2$.

\begin{thm}\name{thm: twobytwo} Let $\nu$ be a Federer measure on
$\br^d$, $U\subset\br^d$ open,  and $F:U\to M_{2,2}$ a continuous map such that $(\vd \circ F,\nu)$ is 
 good.
  \begin{itemize}
  \item[(a)] Suppose that   $(\vf ,\nu)$ 
is nonplanar for any row or column $\vf$ of $F$; then $F_*\nu$
is strongly extremal.
\item[(b)] Suppose that both $(F,\nu)$ and $(F^T,\nu)$ are row-nonplanar;
then $F_*\nu$ 
is  extremal.
\end{itemize}
\end{thm}

As was mentioned before, $(\vd \circ F,\nu)$ happens to be 
 good 
 when $\nu$ is Lebesgue and functions $f_{ij}$
are real analytic. Thus, in particular, the pushforwards of Lebesgue measure by
$$
x\mapsto \begin{pmatrix} x & x^2\\ x^3 & x^4\end{pmatrix}\quad\text{ or } \quad x\mapsto \begin{pmatrix} x & x^2\\ x^2 & x\end{pmatrix}
$$
are strongly extremal, even though the determinant of the first map is identically zero, 
and the image of the second one is contained in a two-dimensional subspace of $M_{2,2}$ -- and therefore the nonplanarity condition of Theorem \ref{thm: strexgeneral} is violated in both cases.
Likewise, the pushforwards of Lebesgue measure by  
$
x\mapsto \begin{pmatrix} x & x\\ x & x^2\end{pmatrix}$ or even $
x\mapsto \begin{pmatrix} x & x\\ x & 2x\end{pmatrix}$ are extremal 
(it is clear that strong extremality
fails in the latter cases).
The proof will be an illustration of techniques described in  \S\ref{exp}: we will
estimate from below the norms of  projections of 
$u_{F(\cdot)}\vw$ onto  $E_\vt^+$ uniformly in $\vw$ and $\vt$.

\begin{proof}[Proof of Theorem \ref{thm: twobytwo}] We will be proving  both parts simultaneously,
since they are based on the same  computation. 
We need to look through elements of $\fw_\ell$ where $\ell = 1$, $2$ or $3$. 
Since the assumptions on $F$ are obviously invariant under transposition, the computations for $\ell=1$ and $\ell=3$ are identical, i.e.\ dual to each other (see the remark at the
end of \S\ref{dioph}). Thus the two cases to consider correspond to vectors and bi-vectors $\vw$ respectively. As in \S\ref{nondiv}, we will denote the standard basis of $\R^4$ by $\{\ve_1,\ve_2,\vv_1,\vv_2\}$, so that the $u_Y$-action is described via \equ{actionbasis} and \equ{action}.

First consider $\ell = 1$ and take
$$\vw=a_1\ve_1+a_2\ve_2+b_1\vv_1+b_2\vv_2\in\fw_1\,.$$ 
Note that for any $\vt\in\fa$, $E_\vt^+\cap \bigwedge^1(\R^4)$ is spanned by $\ve_1$ and $\ve_2$,
and therefore
$$
\pi_\vt^+ (u_{F(\vx)}\vw)=\big(a_1+b_1f_{11}(\vx)+b_2f_{12}(\vx)\big)\ve_1+\big(a_2+b_1f_{21}(\vx)+b_2f_{22}(\vx)\big)\ve_2\,.$$
Identifying $\ve_1$ with $\begin{pmatrix}1 \\ 0\end{pmatrix}$ and  $\ve_2$ with $\begin{pmatrix}0 \\ 1\end{pmatrix}$ one can write
$$
\pi_\vt^+ u_{F(\vx)}\vw=F(\vx)\begin{pmatrix}b_1 \\ b_2\end{pmatrix} + \begin{pmatrix}a_1 \\ a_2\end{pmatrix}\,.$$
Since at least one of $a_i$, $b_j$ is a nonzero integer,  the row-nonplanarity of $(F,\nu)$ implies that for any ball $B\subset U$ with $\nu(B) > 0$
there exists $c > 0$ such that norms of the  above vectors, uniformly in $\vw\in\fw_1$ and $\vt\in\fa$, are not less than $c$ for some $\vx\in \supp\,\nu \cap B$. Hence Corollary \ref{cor: uniform lower} applies. 
Note that in this case the weaker assumption of part (b) was sufficient to draw the required conclusion.


For the case $\ell=2$, take
$$\vw=a\ve_1\wedge\ve_2+b_{11}\ve_1\wedge\vv_1+b_{12}\ve_1\wedge\vv_2+b_{21}\ve_2\wedge\vv_1+b_{22}\ve_2\wedge\vv_2+c\vv_1\wedge\vv_2\in\fw_2$$
and write
\eq{112}
{\begin{aligned}u_{F(\cdot)}\vw&=\big(a+b_{11}f_{21}+b_{12}f_{22}-b_{21}f_{11}-b_{22}f_{21}+c \det(F)\big)\ve_{\{1,2\}} \\
     &+(b_{11}-cf_{12})\ve_1\wedge\vv_1+(b_{12}+cf_{11})\ve_1\wedge\vv_2 \\
     &+(b_{21}-cf_{22})\ve_2\wedge\vv_1+(b_{22}+cf_{21})\ve_2\wedge\vv_2+c\vv_{\{1,2\}}\end{aligned}}
First let us describe the argument in case (b). When $\vt\in\fr$, that is, $t_1 = t_2 = t_3 = t_4$,
it is easy to see that all the elements $\ve_i \wedge \vv_{j}$ are in $E_\vt^+$. Let $\pi$ be the orthogonal 
projection of $\bigwedge^2(\R^4)$ onto the span of $\ve_1\wedge\vv_1$ and $\ve_2\wedge\vv_1$.
Identifying these with $\begin{pmatrix}1 \\ 0\end{pmatrix}$ and   $\begin{pmatrix}0 \\ 1\end{pmatrix}$,
one can write
$$
\pi (u_{F(\cdot)}\vw )=F(\cdot)\begin{pmatrix}0 \\ -c \end{pmatrix} + \begin{pmatrix}b_{11} \\ b_{21}\end{pmatrix}\,.$$
Thus the desired estimate holds whenever at least  one of $b_{11},b_{21},c$ is nonzero.
Otherwise, either $b_{12} = \langle u_{F(\cdot)}\vw ,\ve_1\wedge\vv_2\rangle$ or  $b_{22} = \langle u_{F(\cdot)}\vw ,\ve_2\wedge\vv_2\rangle$ or  $a = \langle u_{F(\cdot)}\vw ,\ve_{\{1,2\}}\rangle$ is a nonzero integer, and therefore the estimate  of Corollary \ref{cor: uniform lower} holds in this case as well.

Now turn to part (a). 
 It is not hard to see that for any $\vt\in\fa$, the
dimension
of $E_\vt^+\cap \bigwedge^2(\R^4)$ is at least three. Specifically, let $i_0$ be such that $t_{i_0} = \max_{i = 1,\dots,4}t_i$. If $i_0 \le 2$, then clearly $\ve_{i_0}\wedge \vv_i \in E_\vt^+$, $i = 1,2$, 
and otherwise $\ve_i \wedge \vv_{i_0-2}\in E_\vt^+$, $i = 1,2$. In addition, $\ve_{\{1,2\}}$ is clearly 
also always in $ E_\vt^+$. Without loss of generality let us assume that $i_0 = 1$ (the other cases
are treated similarly). Then both $\ve_1\wedge\vv_1$ and $\ve_1\wedge\vv_2$ belong to $E_\vt^+$,
so whenever  at least  one of $b_{11},b_{12},c$ is nonzero, the nonplanarity of $\big((f_{11}, f_{12}),\nu\big)$
implies the desired estimate.  Otherwise, the projection of $u_{F(\cdot)}\vw$ onto $\ve_{\{1,2\}}$ is equal to $a-b_{21}f_{11}-b_{22}f_{21}$, and one of $b_{21},b_{22},a$ is definitely nonzero; therefore the nonplanarity of  $\big((f_{11}, f_{21}),\nu\big)$ applies and finishes the proof.
\end{proof}

We would like to point out that the nonplanarity conditions of the above theorem
are as close to being optimal as the standard nonplanarity assumption on the pair $(\vf,\nu)$
in the case $\min(m,n) = 1$. Indeed, if the nonplanarity of   $(\vf,\nu)$ is 
violated by the existence of a nontrivial {\it integer\/} linear combination of $1,f_1,\dots,f_n$
vanishing on $B\,\cap \,\supp\,\nu$, then clearly every point of $\vf(B\,\cap \,\supp\,\nu)$ is \vwa.
Likewise, if a nontrivial integer linear combination of $1$ and the components of 
some row or   column of $F$ 
vanishes on $B\,\cap \,\supp\,\nu$, then  $\vf(B\,\cap \,\supp\,\nu)$   consists of VWMA matrices.
Indeed, if the above is the case for one of the rows of $F$, then $\Pi\big(F(\cdot)\vq + \vp\big) \equiv 0$
on $B\,\cap \,\supp\,\nu$;  the same conclusion for one of the columns
follows from the transference principle. Similarly,  if the row-nonplanarity assumption is 
violated by the existence of a nonzero integer vector $\vq\in\Z^n$ such that
the 
restriction of the map
$\vx\mapsto F(\vx) \vq$ to $B\,\cap \,\supp\,\nu$
is  a constant integer $\vp\in\Z^m$, then $ F(\cdot)\vq - \vp  \equiv 0$  on $B\,\cap \,\supp\,\nu$, 
and hence obviously $F(B\,\cap \,\supp\,\nu)$   consists of VWA matrices.

Of course there is a gap between vanishing of {\it all\/} linear combinatinons and non-vanishing
of a non-trivial  {\it integer\/} linear combination; a precise criterion (in the class of Federer measures
and good pairs) is likely to involve some \di\ conditions on the parameterizing coefficients of
the smallest affine subspace containing the image of $F$, similarly to the results of 
\cite{gafa, dima tams, yuqing},

\medskip

Looking at Theorem \ref{thm: twobytwo} one may wonder whether or not it is possible 
in general to derive
strong extremality or at least extremality of $F_*\nu$ 
from conditions involving just linear combinations
of rows/columns of $F$. This turns out not to be the case when $\max(m,n) > 2$. 
Indeed, as we have seen in Corollary \ref{cor: uniform zero}, linear growth of 
$g_\fr \Lambda_Y$  is implied by   vanishing of the projection of $ u_Y\vw$ for some $\vw\in\fw$
onto the space $E^+_\vt$, where $\vt\in\fr$ is arbitrary. Next we are going to show that such vanishing 
conditions can boil  down to higher degree polynomial relations between the columns (or rows) of $Y$.

For simplicity consider the case $n = 2$ (similarly one can treat the general case). Fix $t > 0$ and
$\vt =  (\frac tm, \ldots,
\frac tm, \frac t2,  
\frac t2)$, and observe that  \eq{span}{ \textstyle \bigwedge ^2 \,(\R^{m+2})\cap E^+_\vt\text{   is spanned by }\ve_i\wedge \ve_j, \ 
1\le j < j \le m\,.} Indeed, unlike   the case $m = 2$ considered in Theorem \ref{thm: twobytwo}, elements $\ve_i\wedge \vv_j$ are contracted by $g_\vt$, namely one has $g_\vt(\ve_i\wedge \vv_j) = e^{t(\frac 1n - \frac 12)}\ve_i\wedge \vv_j$. Denote by $\vy_1,\vy_2$ the columns of $Y$. 
Now take an arbitrary $\vw\in\fw_2$, and denote
by $W$ the  plane in $\R^{m+2}$ corresponding to $\vw$. Also denote 
 by $V$ the plane spanned by $\vv_1,\vv_2$ and by $E$ the span of $\{\ve_i : i = 1,\dots,m\}$. 
Clearly the following three cases can occur: the orthogonal projection of $W$ onto $V$ can have
dimension $1$, $2$ or $0$. In the latter case $\vw$ belongs to $E^+_\vt$ and is $u_Y$-invariant, 
therefore $\pi_\vt^+ u_Y\vw$ does not vanish. The other two cases are more interesting.
\medskip

\noindent{\bf Case 1.} If $W$ projects  onto a one-dimensional subspace of $V$, one can write
$\vw = \vv \wedge \vu$ where $\vv$, $\vu$ are nonzero integer vectors in $V$ and $E$ respectively.
In other words (identifying $E$ with $\R^m$ as before), 
$$ \vw = \vu\wedge (a\vv_1 + b\vv_2), \text{ where }\vu\in \Z^m\nz, \ (a,b)\in\Z^2\nz\,.$$
From \equ{actionbasis} and \equ{span} it then follows  that 
 $$
  \pi_\vt^+u_Y\vw = \pi_\vt^+ \vu\wedge \big(a(\vv_1+ \vy_1) + b(\vv_2 + \vy_2)\big) =  \vu\wedge (a \vy_1 + b \vy_2)\,.$$
\noindent{\bf Conclusion 1:} if a nontrivial integer linear combination of columns of $Y$ is proportional
to an integer vector, then $Y$ is very well approximable. In particular, if this happens for $Y = F(\x)$
with the coefficient of proportionality being a function of $\x$, then $F(\x)$ is VWA for every $\x$.
Consider for example 
$F(x) =  \left(
\begin{array}{ccccc}
x & x^2 + x^3 \\
x^2 & x + x^3 \\
x^3 & x + x^2
\end{array}\right)$.
Each  row (resp.,  column) of $F$ is a nondegenerate
polynomial map $\br\to\br^2$ (resp.,  $\br\to\br^3$). However $ F(x)$ is \vwa\ for every $x$,
since the sum of its columns is equal to $(x + x^2 + x^3)(\ve_1 + \ve_2 + \ve_3)$.
\medskip

\noindent{\bf Case 2.} In the generic situation, when the plane $W$ projects surjectively onto $V$,
using Gaussian reduction over integers, one can express $\vw$ 
as
$$
 \vw =  ( \vu_1 +a \vv_1)\wedge (\vu_2 + b\vv_2), \text{ where }\vu_1,\vu_2\in \Z^m\nz,\, a,b\in\Z\nz\,.$$
Then
 $$
  \pi_\vt^+u_Y\vw =   \pi_\vt^+ \big( \vu_1 + a(\vv_1+  \vy_1)\big)\wedge \big(\vu_2 + b(\vv_2 +  \vy_2)\big) =      ( \vu_1 + a\vy_1)\wedge (\vu_2  + b\vy_2)\,.$$
  \noindent{\bf Conclusion 2:} if a integer translate  of an integer multiple of a column  of $Y$ is proportional
to an integer  translate  of an integer multiple of the other column, then $Y$ is very well approximable.
For example matrices
$$F_1(x) =  \left(
\begin{array}{ccccc}
x & x^4 \\
x^2 & x^5 \\
x^3 & x^6
\end{array}\right)\text{ and } \ F_2(x) =   \left(
\begin{array}{ccccc}
x\  & 2x^2 + 3x \\
x^2\  & 2x^3 +2x^2 - x \\
x^3\  & 2x^4 +2x^3 + x
\end{array}\right)$$
are VWA for every $x$ (even though, as in the previous example, their rows and columns   are nondegenerate
polynomial maps). This is completely clear as far as $F_1$  is concerned -- its columns  are proportional. 
However it is far less obvious to understand the reason for the non-extremality of $(F_2)_*\lambda$,
namely, that $$2\begin{pmatrix}
x \\
x^2 \\
x^3
\end{pmatrix} + \begin{pmatrix}
1 \\
-1 \\
1
\end{pmatrix}\text{ and }\begin{pmatrix}
2x^2 + 3x  \\
2x^3 +2x^2 - x \\
2x^4 +2x^3 + x
\end{pmatrix} + \begin{pmatrix}
1 \\
-1 \\
1
\end{pmatrix}\text{ are proportional.} 
$$
It appears to be a challenging task to devise an algorithm which detects all the aforementioned
obstructions to extremality, say for matrices whose elements are integer polynomials in one real variable\footnote{Arguing similarly to the proof of  Theorem \ref{thm: twobytwo}, it is possible to show that the obstructions listed in Cases 1 and 2 above, together with the  linear
ones taken care of by assuming the row-nonplanarity of $F$ and $F^T$, can be used to generate a complete 
list of obstructions within the class of Federer measures and good pairs when $n = 2$ and $m = 3$. 
For higher dimensions the situation is more complicated, 
that is, one can produce 
non-trivial obstructions by considering $u_{Y}$-action on $\fw_p$ for $p\ge 3$.}. This is part of a vague general problem, asked in \cite[\S 9.1]{Gorodnik},  to describe general 
conditions which are sufficient for extremality or strong extremality and are `close to being optimal', in the sense of the  discussion after Theorem \ref{thm: twobytwo},  within certain class of maps. 
(The latter theorem, incidentally,  settles the problem for $m = n = 2$ in the class of Federer measures
and good pairs.) This circle of problems will be addressed in a forthcoming paper \cite{BKM}.

\section{Concluding remarks and open questions}
\name{concl}

\subsection{Improving Dirichlet's Theorem}\name{dt}
Another application of techniques developed in  this paper yields a generalization of a  theorem from \cite{KW}, which in its turn has generalized many earlier results.
The starting point for the general set-up of the problem is a multi-parameter form of Dirichlet's Theorem\footnote{In \cite{nimish mult} it was
referred to as Dirichlet-Minkowski Theorem.}: for any system  of  linear forms $Y_1,\dots,Y_m$
(rows of \amr)
and for any $\,\vt \in\fa
$
there exist solutions $\vq  = (q_1,\dots,q_n)\in \Z^n\nz$ and 
$\vp = (p_1,\dots,p_m) \in \Z^m$ 
of
\eq{mdt}{
\begin{cases}
|Y_i\vq - p_i| < e^{-t_i}\,,\quad &i = 1,\dots,m
 \\  
\ \ |q_j| \le e^{t_{m+j}}\,,\quad &j = 1,\dots,n
\,.
\end{cases}}
Then, given an unbounded subset $\mathcal{T}$ of $\fa
$
and positive $\vre < 1$, one says that
{\sl Dirichlet's Theorem can be $\vre$-improved for $Y$  along\/} $\mathcal{T}$, 
or $Y\in\DI_\vre(\mathcal{T})$,  
 if there is $T$ such that for every  $\vt =
(t_1,\dots,t_{m+n})\in\mathcal{T}$ with $t > T$, 
the  inequalities 
\eq{mdtw}{
\begin{cases}
|Y_i\vq - p_i| < \vre e^{-t_i}\,,\quad &i = 1,\dots,m
 \\  
\ \ |q_j| < \vre e^{t_{m+j}}\,,\quad &j = 1,\dots,n
\,,
\end{cases}
}
i.e., \equ{mdt} with the right hand side terms 
 multiplied by $\vre$,
 have nontrivial integer solutions. 
Using 
an elementary argument  
dating back to  Khintchine, one can show that 
for any $m,n$ and any unbounded $\mathcal{T}\subset\fa
$,  
$\DI_\vre(\mathcal{T})$ has Lebesgue 
measure zero as long as $\vre < 1/2
$. In \cite{KW} a similar statement was proved   for pushforwards of Federer  measures
to  $\R^n \cong M_{1,n}$ by continuous maps $\vf$. Namely, 
let 
$\nu$ be a  $D$-Federer 
measure  on
$\R^d$, $U\subset\br^d$ open, and   $\vf:U\to\br^n$  continuous such that 
the pair $(\vf,\nu)$ is 
$(C,\alpha)$-good and
nonplanar. Then it was proved in \cite[Theorem 1.5]{KW} that $ \vf_*\nu\big(\DI_\vre(\mathcal{T})\big) = 0$
for any unbounded $ 
\mathcal{T}\subset\fa$ and any $\vre < \vre_0$, where
$\vre_0$ depends only on $d,n,C,\alpha,D$. Note that here one needs a uniform version of 
the definition of a good pair: $(\vf,\nu)$ is said to be {\sl
\cag\/}   if for $\nu$-a.e.\ $x$
   there exists  a
neighborhood $U$ of $x$  
such that $(\vf,\nu)$ is 
 $(C,\alpha)$-good  on $U$. 
 
 We refer the reader to \cite{KW} and \cite{nimish} for a history of the 
 subject, which had been initiated in \cite{Davenport-Schmidt, Davenport-Schmidt2} for the case $\ft = \fr$, that is, dealing with Dirichlet's Theorem in its classical form. Also note that recent results of Shah \cite{nimish, nimish mult} show that in many cases, with $\nu = \lambda$ and $\vf$ real analytic, a similar
 result holds with $\vre_0 = 1$. 

\medskip
It turns out that a combination of methods of  \cite{KW} and the present paper can produce 
the following generalization to the case $\min(m,n) > 1$:

\begin{thm}\name{thm: digeneral} For any\,
$d,m,n\in\N$ and   
$\,C,\alpha,D > 0$ there exists
$\vre_0$
with the following property.
 Let $U$ be an open subset of
$\R^d$,     $F:U\to\mr$ continuous and
 $\nu$  a  
measure  on
$U$. 
Assume that $\nu$ is  $D$-Federer,  and $(\vd \circ F,\nu)$ is 
 $(C,\alpha)$-good and
 nonplanar.
Then 
$F_*\nu\big(\DI_\vre(\mathcal{T})\big) = 0$ 
 for any unbounded 
$\mathcal{T}\subset\fa$ and any $\vre < \vre_0$.
\end{thm}

It can be shown, by combining the argument of \S\ref{indepvar} with \cite[Proposition 4.4]{KW},
that a uniform version of Proposition \ref{prop: indepvar} holds; that is,
for $\nu = \lambda$ and $F$ as in \equ{f} one can choose $C,\alpha$ such that
 $(\vd\circ F,\nu)$ is \cag; thus 
 the conclusion of the above theorem holds for $F_*\lambda$ as in Theorem 
  \ref{thm: strex}, with some positive $\vre_0$. 
  Details and further results along these lines will appear in a forthcoming paper. 
  It seems natural to conjecture that for $\nu = \lambda$ and real analytic $F$ such that
  $\vd\circ F$ is nonplanar an analogue of Shah's result holds, that is,  sets
  $\DI_\vre(\mathcal{T})$ are $F_*\lambda$-null for any $\vre < 1$.

\subsection{Inhomogeneous \di\ problems}\name{inhom}
A method allowing to transfer results on extremality and strong extremality of measures
on $\mr$ to inhomogeneous \da\  has been recently developed by Beresnevich and Velani in \cite{BV}. In the  inhomogeneous set-up, instead
of systems of linear forms given by \amr, one considers systems of {\it affine\/} forms $(Y,\vz)$, 
that is, maps $\vq\mapsto Y\vq + \vz$ where \amr\ and $\vz\in\R^m$. Generalizing the homogeneous setting by identifying $Y$ with $(Y,0)$, let us say that 
$(Y,\vz)$ is VWA if for some $\delta > 0$  there are infinitely many $\vq\in \Z^n$
such that $$\|
Y\vq + \vz - \vp\| < \|\vq\|^{ - n /
m - \delta} \text{ for some }\vp\in\Z^m\,, $$
and that it is VWMA if for some $\delta > 0$ there are infinitely many $\vq\in \Z^n$
such that $$\Pi(Y\vq  + \vz -\vp) < \Pi_{+}(\vq)^{-(1+\delta)} \text{ for some }\vp\in\Z^m\,.$$
From the Borel-Cantelli Lemma it is clear
 that for any $\vz\in\R^m$ the set
$${\rm VWMA}_\vz\df \{ Y\in\mr : (Y,\vz)\text{ is VWMA}\} $$
 has zero Lebesgue measure; and, since VWA obviously implies VWMA, the same is true for
$${\rm VWA}_\vz\df \{ Y\in\mr : (Y,\vz)\text{ is VWA}\}\,. $$
Following \cite{BV}, let us say that a measure $\mu$ on $\mr$ is {\sl inhomogeneously extremal\/} (resp.,  {\sl inhomogeneously strongly extremal\/})
if $\mu(\rm{VWA}_\vz) = 0$ for any $\vz\in\R^m$    (resp.,  $\mu(\rm{VWMA}_\vz) = 0\ \forall\,\vz\in\R^m$). 

One of the main results of  \cite{BV} is the following transference phenomenon: under some 
regularity conditions on $\mu$, the inhomogeneous  properties defined above are equivalent to their
(apriori weaker) homogeneous analogues. Specifically, Beresnevich and Velani define  the class of measures
on $\mr$ which they call {\sl contracting almost everywhere} and a subclass of measures
 {\sl strongly contracting  almost everywhere} (we refer the reader  to \cite{BV} for precise definitions). According to \cite[Theorem 1]{BV},  a (strongly) contracting  almost everywhere measure on $\mr$ 
is  (strongly) extremal if and only if it is inhomogeneously  (strongly) extremal. 
 Using this, \cite{BV} establishes inhomogeneous strong extremality of many measures proved earlier to be strongly extremal, such as
$\vf_*\lambda$ where $\vf$ is as in Theorem \ref{thm: strex}, or, more generally, arbitrary friendly measures  on $\R^n$.

As remarked at the end of \cite{BV}, `any progress on the homogeneous extremality 
problem can be transferred over to the inhomogeneous setting'. Indeed, many measures on $\mr$
discussed in the present paper can be shown to be strongly contracting almost everywhere.
Here is an example:
 suppose that for any $i = 1,\dots, m$ we are given a contracting measure $\mu_i$ on 
$\R^n$, where the latter space is  identified with the space of $i$th rows of  \amr. Then  it is clear from the definitions  that 
$\mu_1\times\dots\times \mu_m$ is strongly contracting. Therefore Theorem \ref{thm: strexanddi}
and the results of \cite{BV} imply

\begin{thm}\name{thm: inhstrex} Let $F$ be as in Theorem \ref{thm: strexanddi}. Then $F_*\lambda$
 is 
inhomogeneously   strongly extremal.
 \end{thm}

This motivates a problem of checking   contracting  and strongly contracting  properties of 
other measures on $\mr$ proved in the present paper to be extremal or strongly extremal. 
For example, it would be interesting to understand under what conditions on a smooth submanifold
of $\mr$ its Riemannan volume measure is (strongly) contracting (Theorem 4 of \cite{BV}
deals with the case  $n = 1$).

\subsection{What is next?}\name{other} Here is an incomplete list of other possible directions for further research:

\subsubsection{} Can one characterize extremal or strongly extremal affine subspaces of $\mr$
in the spirit  of \cite{gafa} which settled the problem for $\min(m,n) = 1$? Or, more generally, subspaces with a given \de\ following \cite{dima tams}?

\subsubsection{} Is it possible to obtain some Khintchine-type results for smooth submanifolds of $\mr$ with
$\min(m,n) > 1$? That is, study inequalities of type \equ{vwa} with a power  of the  norm of $ \vq $ in the right hand side
replaced by a general non-increasing function of $\|\vq\|$ satisfying the convergence or divergence conditions of the
Khintchine-Groshev Theorem. 
 Note that, as of now,  the convergence case of the problem is not settled even for nondegenerate submanifolds
of $M_{m,1}$ where $m > 2$; however  recent divergence theorems  \cite{B, BDV}  gives a hope of 
possible extensions to curves in the space of matrices. Likewise,  convergence-type results
of \cite{VV} for planar curves give a hope for a complete Khintchine-type theorem for smooth 
`sufficiently nondegenerate' (in the spirit of Theorem \ref{thm: twobytwo}) smooth curves in $M_{2,2}$.

\subsubsection{} Following \cite{KT} and \cite{Anish}, it should not be difficult to extend the results of the present paper
to metric \di\ problems over non-Archimedean local fields and their products.

\enlargethispage{3\baselineskip}
\end{document}